\documentclass[a4paper,oneside,11pt]{amsart}

\reversemarginpar

\usepackage[colorlinks,hyperindex,linkcolor=blue,urlcolor=black,pdftitle={Tests for complete K-spectral sets},pdfauthor={Michael A. Dritschel, Daniel Estevez, and Dmitry Yakubovich}]{hyperref}

\usepackage{amsmath}
\usepackage{amsfonts}
\usepackage{amssymb}
\usepackage{amsthm}

\usepackage[english]{babel}
\usepackage{enumerate}

\usepackage{graphicx}
\usepackage{color}

\usepackage[abbrev]{amsrefs}

\newcommand{\T}{\mathbb{T}}
\newcommand{\D}{\mathbb{D}}
\newcommand{\C}{\mathbb{C}}
\newcommand{\R}{\mathbb{R}}
\newcommand{\N}{\mathbb{N}}

\newcommand{\B}{\mathcal{B}}

\newcommand{\A}{A}

\newcommand{\Cont}{\mathcal{C}}

\renewcommand{\Re}{\operatorname{Re}}

\newcommand{\id}{\operatorname{id}}
\newcommand{\cb}{\text{cb}}

\newcommand{\Rat}{\operatorname{Rat}}

\theoremstyle{plain}
\newtheorem{theorem}{Theorem}
\newtheorem{lemma}[theorem]{Lemma}
\newtheorem{corollary}[theorem]{Corollary}

\newtheorem{theoremcite}{Theorem}

\newtheorem{theoremarticle}{Theorem}
\newtheorem{lemmaarticle}[theoremarticle]{Lemma}

\theoremstyle{remark}
\newtheorem*{remark}{Remark}

\theoremstyle{definition}
\newtheorem*{definition}{Definition}
\newtheorem*{definitions}{Definitions}
\newtheorem{condition}{Condition}

\newtheoremstyle{named}{}{}{\itshape}{}{\bfseries}{.}{.5em}{\thmnote{#3 }#1}
\theoremstyle{named}
\newtheorem*{problem}{Problem}
\newtheorem*{question}{Question}

\newcommand{\al}{\alpha}
\newcommand{\de}{\delta}
\newcommand{\si}{\sigma}
\newcommand{\la}{\lambda}
\newcommand{\Om}{\Omega}
\newcommand{\tht}{\theta}
\newcommand{\wt}{\widetilde}

\renewcommand{\phi}{\varphi}

\newcommand{\prt}{\partial}
\newcommand{\ovl}{\overline}
\newcommand{\sm}{\setminus}
\renewcommand{\epsilon}{\varepsilon}
\newcommand{\eps}{\varepsilon}

\begin{document}

\title{Tests for complete $K$-spectral sets}

\author{Michael A. Dritschel}
\address{
  School of Mathematics and Statistics\\
  Herschel Building\\
  University of Newcastle\\
  Newcastle upon Tyne\\
  NE1 7RU\\
  UK}
\email{michael.dritschel@newcastle.ac.uk}
\author{Daniel Est\'evez}
\address{Departamento de Matem\'{a}ticas\\
  Universidad Aut\'onoma de Madrid\\ Cantoblanco 28049 (Madrid)\\
  Spain}
\email{daniel.estevez@uam.es}
\author{Dmitry Yakubovich}
\address{Departamento de Matem\'{a}ticas\\
  Universidad Aut\'onoma de Madrid\\ Cantoblanco 28049 (Madrid)\\ Spain\\
  and Instituto de Ciencias Matem\'{a}ticas (CSIC - UAM - UC3M - UCM)}
\email{dmitry.yakubovich@uam.es}

\thanks{The second author has been supported by a grant of the
  Mathematics Department of the Universidad Aut\'onoma de Madrid and
  the Project MTM2015-66157-C2-1-P of the Ministry of Economy and
  Competitiveness of Spain. The third author has been supported by the
  Project MTM2015-66157-C2-1-P and by the ICMAT Severo Ochoa project
  SEV-2015-0554 of the Ministry of Economy and Competitiveness of
  Spain and the European Regional Development Fund (FEDER).}



\begin{abstract}
  Let $\Phi$ be a family of functions analytic in some neighborhood of
  a complex domain $\Omega$, and let $T$ be a Hilbert space operator
  whose spectrum is contained in $\overline\Om$.  Our typical result
  shows that under some extra conditions, if the closed unit disc is
  complete $K'$-spectral for $\phi(T)$ for every $\phi\in \Phi$, then
  $\overline\Om$ is complete $K$-spectral for $T$ for some constant
  $K$.  In particular, we prove that under a geometric transversality
  condition, the intersection of finitely many $K'$-spectral sets for
  $T$ is again $K$-spectral for some $K\ge K'$.  These theorems
  generalize and complement results by Mascioni, Stessin, Stampfli,
  Badea-Beckerman-Crouzeix and others.  We also extend to non-convex
  domains a result by Putinar and Sandberg on the existence of a skew
  dilation of $T$ to a normal operator with spectrum in
  $\partial\Omega$.  As a key tool, we use the results from our
  previous paper~\cite{article1} on traces of analytic uniform
  algebras.
\end{abstract}

\maketitle

\section{Introduction}

Let $T$ be an operator on a Hilbert space $H$ and $\Om$ a bounded
subset of $\C$ containing the spectrum $\si(T)$.  We recall that,
given a constant $K\ge 1$, the closure $\ovl\Om$ of $\Om$ is said to
be \textit{a complete $K$-spectral set for $T$} if the matrix von
Neumann inequality
\begin{equation}
  \label{eq:complete-k-spectral}
  \|p(T)\|_{\B(H\otimes \C^s)}
  \le
  K\, \max_{z\in \ovl\Om}
  \|p(z)\|_{\B(\C^s)}
\end{equation}
holds for any square $s\times s$ rational matrix function $p(z)$
of any size $s$ and with poles off of $\overline{\Om}$; here
$\B(H)$ denotes the space of linear operators on $H$.
The set
$\overline{\Omega}$ is called a \emph{$K$-spectral set for $T$} if
(\ref{eq:complete-k-spectral}) holds for $s = 1$.
By a well-known theorem of Arveson~\cite{Arveson}, $\ovl\Om$ is a
complete $K$-spectral set for $T$ for some $K\ge 1$ if and only if
$T$ is similar to an operator, which has a normal dilation $N$
with $\si(N)\subset \partial \Om$; the importance of complete
$K$-spectral sets is due to this result.

We denote by $\widehat{\C}$ the Riemann sphere $\widehat{\C} = \C \cup
\{\infty\}$.  By a \emph{Jordan domain} in $\widehat{\C}$ we mean an
open domain $\Omega\subset\widehat{\C}$ whose boundary is a Jordan
curve.  A Jordan domain (or Jordan domain in $\C$) is just a bounded
Jordan domain in $\widehat{\C}$.  A curve $\Gamma \subset \C$ is
called \emph{Ahlfors regular} if $|B(z,\varepsilon) \cap \Gamma| \leq
C\varepsilon$, for every $\varepsilon > 0$ and every $z \in \Gamma$,
where $C$ is a constant independent of $\varepsilon$ and $z$.  Here
$|\cdot|$ denotes the arc-length measure and $B(z,\varepsilon)$ is the
open disk of radius $\varepsilon$ and center $z$.

By a \emph{circular sector with vertex} $z_0$ we mean a set in $\C$ of
the form
\begin{equation*}
  \{z\in\C : 0<|z-z_0|<r,\ \alpha < \arg z < \beta \},
\end{equation*}
where $r > 0$ and $\alpha,\beta\in\R$, $0 < \beta - \alpha < 2\pi$.
The aperture of such a circular sector is the number $\beta - \alpha$.

If $\Omega_1,\Omega_2 \subset \widehat{\C}$ are two open sets, and
$\infty \neq z_0 \in \partial\Omega_1\cap \partial\Omega_2$ is a point
in the intersection of their boundaries, we say that the boundaries of
$\Omega_1$ and $\Omega_2$ \textit{intersect transversally at} $z_0$ if
one can find five pairwise disjoint circular sectors
$S_0,S_1^l,S_1^r,S_2^l,S_2^r$ with vertex $z_0$, having the same
aperture, and such that the following conditions are satisfied:
\begin{itemize}
\item $S_0$ does not intersect $\overline{\Omega}_1 \cup
  \overline{\Omega}_2$.
\item $B(z_0,\varepsilon) \cap \partial\Omega_j \subset S_j^l\cup
  S_j^r\cup\{z_0\}$ for $j=1,2$ and some $\varepsilon > 0$.
\item For every $\delta > 0$, $B(z_0,\delta)\cap\Omega_1\cap\Omega_2$
  is not empty.
\end{itemize}
In the case when $\infty \in \partial\Omega_1\cap\partial\Omega_2$, we
say that the boundaries of $\Omega_1$ and $\Omega_2$ intersect
transversally at $\infty$ if the boundaries of $\psi(\Omega_1)$ and
$\psi(\Omega_2)$ intersect transversally at $0$, where $\psi(z) =
1/z$.  We say that the boundaries of $\Omega_1$ and $\Omega_2$
intersect transversally if they intersect transversally at every point
of $\partial\Omega_1\cap\partial\Omega_2$.  Note that the third
condition in the definition of a transversal intersection implies that
$\overline{\Omega_1 \cap \Omega_2} = \overline{\Omega}_1 \cap
\overline{\Omega}_2$.

By an analytic arc in $\C$ we mean an image of the interval $[0,1]$
under a function, analytic in its neighborhood. A piecewise analytic
curve will mean a curve which can be subdivided into finitely many
analytic arcs.

We can now state some of the main results of the paper.

\begin{theorem}
  \label{Hav-Nersess-opers}
  Let $\Omega_1, \dots, \Omega_n$ be open sets in $\widehat{\C}$ such
  that the boundary of each set $\Omega_k$, $k = 1,\ldots,n$, is a
  finite disjoint union of Jordan curves.  We also assume that the
  boundaries of the sets $\Omega_k$, $k=1,\ldots,n$, are Ahlfors
  regular and rectifiable, and intersect transversally.  Put $\Omega =
  \Omega_1 \cap\dots \cap \Omega_n$.  Suppose that $T \in \B(H)$, and
  $\sigma(T) \subset \overline{\Omega}$.  There is a constant $K'$
  such that
  \begin{enumerate}
  \item[$(i)$] if each of the sets $\overline{\Omega}_j$ , $j =
    1,\ldots, n$, is $K$-spectral for $T$, then $\overline{\Omega}$ is
    also $K'$-spectral set for $T$; and
  \item[$(ii)$] if each of the sets $\overline{\Omega}_j$ , $j =
    1,\ldots,n$, is complete $K$-spectral for $T$, then
    $\overline{\Omega}$ is a complete $K'$-spectral set for $T$.
  \end{enumerate}
  In both cases, $K'$ depends only on the sets $\Omega_1, \dots,
  \Omega_n$ and the constant $K$, but not on the operator $T$.
\end{theorem}

As will be seen from the proof, the Ahlfors regularity condition can
be weakened, by requiring that it hold only in some neighborhoods of
the intersection points of the boundary curves $\partial\Omega_j$.

The results of Theorem~\ref{Hav-Nersess-opers} can be viewed as a
generalization of the so called surgery of $K$-spectral sets.  The
articles \cites{Lewis, Stampfli, Stampfli2} are devoted to this topic.
In the case when the sets that one is dealing with are Jordan domains
and their boundaries intersect transversally, the results of these
articles can be obtained as a particular case of
Theorem~\ref{Hav-Nersess-opers}.

In \cite{BadeaBeckermannCrouzeix}, Badea, Beckermann and Crouzeix
prove that the intersection of complete spectral sets which are disks
on the Riemann sphere is a complete $K'$-spectral set (see
Theorem~\ref{badea} below for the precise statement).
Theorem~\ref{Hav-Nersess-opers} is a generalization of this result in
two ways.  Firstly, it allows for the sets $\overline{\Omega}_j$ to be
complete $K$-spectral sets instead of complete spectral sets.
Secondly, it allows for the sets ${\Omega}_j$ to be open sets with
some conditions on the boundary, rather than just that they be disks.
The points of \cite{BadeaBeckermannCrouzeix} that are not covered by
Theorem~\ref{Hav-Nersess-opers} is that there they do not need
transversality and obtain a value of $K$ which is an explicit
universal constant depending only on the number of disks.
Unfortunately, in our results we do not have an explicit control on
the constant $K$.

The particular case $K=1$ is important. Let us say that a domain
$\Omega$ has the rational dilation property if whenever
$\overline{\Omega}$ is a $1$-spectral set for $T$, then
$\overline{\Omega}$ is also a complete $1$-spectral set for $T$.  It
follows from the fact that every contraction has a unitary dilation
that $\D$ has the rational dilation property.  In \cite{Agler85},
Agler proved that every annulus has the rational dilation property.
In general, a domain $\Omega$ with two or more holes does not have the
rational dilation property.  Dritschel and McCullough in
\cite{DritschelMcCullough05} and Agler, Harland and Raphael in
\cite{AHR} found independently examples of domains with two holes
which do not have the rational dilation property.  See also the
article \cite{Pickering} by Pickering, where he shows that no
symmetric domain with two or more holes has the rational dilation
property.

In the next theorem, we deal with open sets satisfying a certain
regularity condition.  If $\Omega \subset \widehat{\C}$ is an open set
with $\infty\notin\partial\Omega$ and $R>0$, we say that $\Omega$
satisfies the \emph{exterior disk condition with radius $R$} if for
every $\lambda \in \partial \Omega$ there is $\mu \in \C$ such that
the open disk $B(\mu,R)$ \emph{touches $\Omega$ at $\lambda$}; that is
$|\lambda-\mu|=R$ and $B(\mu,R) \cap \Omega = \emptyset$.

In order to simplify the geometrical arguments, we will also assume
that $\Omega$ satisfies the following technical condition.

\begin{condition}
  \label{condA}
  There exists a finite collection of closed arcs
  $\{\gamma_k\}_{k=1}^N \subset \partial\Omega$ which cover $\partial
  \Omega$ and intersect at most in their endpoints,
  radii $R_k$, $k=1,\ldots,N$,
  and maps $\mu_k:
  \gamma_k \to \C$, such that for every $\lambda \in \gamma_k$,
  the disk
  $B(\mu_k(\lambda),R_k)$ touches $\Omega$ at $\lambda$ and
  $\bigcap_{\lambda\in \gamma_k} B(\mu_k(\lambda), R_k) \neq \emptyset$.
  We also assume that if $\gamma_k$ and $\gamma_l$ intersect at their
  common endpoint $z_0$, then they do so transversally: that is, there
  are disjoint circular sectors $S_k$ and $S_l$ with vertex $z_0$ such
  that $\gamma_k \subset S_k \cup \{z_0\}$ and $\gamma_l \subset S_l
  \cup \{z_0\}$.
\end{condition}

If $\partial\Omega$ is piecewise $C^2$ and the exterior angles at its
corners are nonzero, then $\Omega$ clearly satisfies
Condition~\ref{condA}.  Moreover, it is possible to prove that if
$\partial\Omega$ is a finite disjoint union of Jordan curves and
$\Omega$ satisfies the exterior disk condition and an interior cone
condition, then $\Omega$ also satisfies Condition~\ref{condA}.  In
particular, the exterior disk condition is formally weaker than
Condition~\ref{condA}.

\begin{theorem}
  \label{Putinar-Sandb-non-conv}
  Let $T$ be a bounded linear operator and $\Omega \subset
  \widehat{\C}$ an open set whose boundary is a finite disjoint union
  of Jordan curves.  Assume that $\infty\notin\partial\Omega$, that
  $\Omega$ satisfies Condition~\ref{condA} and that $\sigma(T)\subset
  \overline\Omega$.  Furthermore, assume that for every $k = 1,\ldots,
  N$ and every $\lambda\in \gamma_k$ we have $\|(T-\mu_k(\lambda)
  I)^{-1}\|\le R_k^{-1}$.  Then $\overline\Omega$ is a complete
  $K$-spectral set for some $K>0$.
\end{theorem}

It is easy to see that the hypotheses are satisfied (for any $R_k>0$)
if $\Omega$ is a convex Jordan domain and the numerical range of $T$
is contained in $\overline\Omega$.  This case was first proved by
\phantomsection Delyon and Delyon\label{DelyonDelyon} in
\cite{DelyonDelyon}.  Theorem~\ref{Putinar-Sandb-non-conv} will be
deduced from this and from Theorem~\ref{Hav-Nersess-opers}.  Putinar
and Sandberg gave a different proof of the Delyon-Delyon result in
\cite{PutinarSandb} by constructing a so called normal skew-dilation,
and relate the constant in this result with C.~Neumann's
``configuration constant'' of a convex domain $\Omega$, see
\cite{PutinarSandb}, Proposition~1.  These articles consider only
$K$-spectral sets instead of complete $K$-spectral sets. However, the
arguments used both in \cite{DelyonDelyon} and \cite{PutinarSandb}
imply the existence of a normal operator $N$ on a larger Hilbert space
$K \supset H$ and having $\sigma(N) \subset
\partial\Omega$, and a bounded linear map $\Xi: C(\partial\Omega)\to
C(\partial\Omega)$ such that
\begin{equation*}
  \label{eq:skew}
  f(T) = P_H (\Xi(f))(N)|H,\qquad f \in \Rat(\overline{\Omega}).
\end{equation*}

It follows from Lemma~\ref{lemma-cb} below that the map $\Xi$ is
completely bounded (see also Crouzeix~\cite{Crouzeix}).  Therefore,
\eqref{eq:skew} implies that $\overline{\Omega}$ is a complete
$K$-spectral set for $T$, and so that under the assumptions of the
Delyon-Delyon theorem, $T$ is similar to an operator having a normal
dilation to $\partial\Omega$.

It is also known that Theorem~\ref{Putinar-Sandb-non-conv} is valid if
$\Omega$ is the unit disk.  In fact, by results of Sz.-Nagy and Foias,
if the hypotheses hold in this case, then $T$ is a $\rho$-contraction
for some $\rho<\infty$ and hence is similar to a contraction.
Therefore Theorem~\ref{Putinar-Sandb-non-conv} can be considered as a
generalization of both of the above mentioned results.  We refer to
Section~\ref{proofs12} for a further discussion and some consequences
of this result.

We will deduce the first part of Theorem~\ref{Hav-Nersess-opers} from
results of Havin, Nersessian and Cerd\`a, where they give various
geometric conditions on domains $\Omega_1, \dots, \Omega_s$ in $\C$,
guaranteeing that every function $f$ in $H^\infty(\cap_j \Omega_j)$
admits a representation
\begin{equation*}
  f=f_1+f_2+\dots+f_s, \quad f_j\in H^\infty(\Omega_j),\ j=1,2, \dots,
  s.
\end{equation*}
In other words, the domains $\Omega_1, \dots, \Omega_s$ admit a
\emph{separation of singularities}.

The proof of the second part of Theorem~\ref{Hav-Nersess-opers} will
also use Lemma~\ref{lemma-cb}.  This lemma says, basically, that if
the range of a bounded linear map is commutative, then the map is
automatically completely bounded.  The particular maps we will be
considering have commutative ranges, so this lemma will be important
in our proofs.  The arguments by Havin, Nersessian and Cerd\`a will
also be key here, because they allow us to deal with commutative
algebras of functions (to which Lemma~\ref{lemma-cb} can be applied)
instead of noncommutative algebras of operators.

Suppose now that $\Phi$ is a collection of functions mapping into
$\mathbb D$, such that each of them is analytic on (its own)
neighborhood of $\Omega$.  The rest of the article is devoted to
finding sufficient conditions for complete $K$-spectrality of the form
\begin{equation}
  \label{eq:tests-str}
  \begin{aligned}
    &\exists K': \forall \phi\in \Phi \qquad \overline{\D} \; \text{is a
      complete $K'$-spectral set for } \phi(T) \\
    & \implies \; \exists K: \quad \text{$\ovl\Om$ is a complete
      $K$-spectral set for $T$.}
  \end{aligned}
\end{equation}
Here $\D$ stands for the open unit disk.
Notice that for any $\phi\in
\Phi$, $\phi(T)$ is defined by the Cauchy-Riesz functional calculus.
Our conditions concern the set $\Omega$ and the family $\Phi$, but we
do not impose extra conditions on $T$.

In particular, a special case of \eqref{eq:tests-str} is that
\begin{equation}
  \label{eq:tests}
  \forall \phi\in \Phi \;\|\phi(T)\|\le 1 \; \implies \; \exists
  K: \text{$\ovl\Om$ is a complete $K$-spectral set for $T$.}
\end{equation}
In this case, $\Phi$ will be called a test collection (a more precise
definition of this notion will be given in the next section).  As we
will show, many known sufficient conditions for complete
$K$-spectrality are easily formulated in the form \eqref{eq:tests} or
in the form \eqref{eq:tests-str} for specific test collections.
Indeed, Theorems~\ref{Hav-Nersess-opers}
and~\ref{Putinar-Sandb-non-conv} can also be given this form if one
uses appropriate Riemann mappings for the test functions (see
Section~\ref{sec:known-criteria} below).

As we will show, implication \eqref{eq:tests-str} holds when we can
solve the following:

\begin{problem}[Algebra Generation]
  \label{prob:alg-gen-prob}
  Suppose $\Phi$ is a finite family of functions in $A(\ovl{\Om})$,
  the algebra of functions in $C(\ovl{\Om})$ that are analytic on the
  interior of $\ovl{\Om}$.  Find geometric conditions guaranteeing
  that $\Phi$ generates $A(\ovl{\Om})$ as an algebra.
\end{problem}

A solution to this was given in our previous article \cite{article1}
(which also was inspired by the techniques of Havin, Nersessian and
Cerd\`a~\cites{HavinNersessian,HavinNersCerda}).  In fact, we more
generally prove that sometimes it is sufficient to show that the
closed subalgebra of $A(\overline{\Omega})$ generated by $\Phi$ is of
finite codimension.

The article is organized in the following manner.  In
Section~\ref{main-results}, we introduce admissible test functions and
we state the main results of this article concerning test collections.
In Section~\ref{sec:known-criteria}, we interpret known criteria of
complete $K$-spectrality in terms of test collections.
Section~\ref{proofs12} contains the proofs of
Theorems~\ref{Hav-Nersess-opers} and \ref{Putinar-Sandb-non-conv},
which were stated in the Introduction.  We also formulate a question,
related with Theorem~\ref{Hav-Nersess-opers}.  In
Section~\ref{sec:article1}, we will list the results of
\cite{article1} that we will need to prove our theorems.
Section~\ref{lemmas} is devoted to the the auxiliary lemmas that are
needed to prove the main results.  Section~\ref{proofs-main} contains
the proofs of the main theorems.  Finally, in
Section~\ref{weakly-admissible}, we treat weakly admissible test
functions, a larger class of test functions for which we can also
prove some results (see, in particular, Theorem~\ref{weak-main}).

\section{Test collections}
\label{main-results}

\subsection{Preliminaries}

We denote by $M_s$ the $C^*$-algebra of complex $s \times s$ matrices.
If $S$ is a (not necessarily closed) linear subspace of a
$C^*$-algebra $\mathcal A$, we denote by $S \otimes M_s$ the tensor
product equipped with the norm inherited from $\mathcal A \otimes
M_s$, which has a unique $C^*$ norm.  One can view $S \otimes M_s$ as
the space of $s \times s$ matrices with entries in $S$.  The simplest
way to norm this is to represent $\mathcal A$ faithfully as a subspace
of $\B(H)$ and then to take the natural norm of $s \times s$ operator
matrices.  If $B$ is another $C^*$-algebra and $\varphi : S \to B$ is
a linear map, we can form the map $\varphi \otimes \id_s : S \otimes
M_s \to B \otimes M_s$, which amounts to applying $\varphi$ entrywise
to $s\times s$ matrices over $S$.  The completely bounded norm of
$\varphi$ is then defined as
\begin{equation*}
  \|\varphi\|_\cb = \sup_{s\geq 1} \|\varphi \otimes \id_s\|.
\end{equation*}

If a compact set $X\subset \mathbb C$ is a complete $K$-spectral set
for a bounded linear operator $T$ and $\Rat(X)$, the algebra of
rational functions with poles off of $X$, is dense in $A(X)$, then the
functional calculus for $T$ extends continuously to $f \in A(X)$, and
we say that such a $T$ admits a continuous $A(X)$-calculus.  Note that
there are various sorts of geometric conditions on $X$ guaranteeing
that $\Rat(X)$ is dense in $A(X)$ (see, for instance,
\cite{Conway}*{Chapter V, Theorem 19.2} for one such).  In particular,
it suffices for $X$ to be finitely connected (see
\cite{Conway}*{Chapter V, Corollary 19.3}).  In what follows, we only
consider finitely connected domains.

\subsection{Different types of test collections}

Here we give the definitions of the several kinds of test collections
used throughout the paper.  As a convenient notation, for $\lambda \in
\widehat{\C} = \mathbb C \cup \{\infty\}$, define $p_\lambda(z) =
(z-\lambda)^{-1}$ if $\lambda \neq \infty$, and $p_\infty(z) = z$.

Assume that $\Omega \subset \widehat{\C}$ is some finitely connected
set.  A \emph{pole set} for $\Omega$ is a finite set $\Lambda \subset
\widehat{\C}\setminus\overline{\Omega}$ that intersects each connected
component of $\widehat{\C}\setminus\overline{\Omega}$.  If $T \in
\B(H)$ and $\sigma(T) \subset \overline{\Omega}$, the
$\Lambda$-\emph{pole size} of $T$ is defined as
$\max_{\lambda\in\Lambda} \|p_\lambda(T)\|$.  We denote the
\emph{$\Lambda$-pole size} of $T$ by $S_\Lambda(T)$.  In the setting
of this article, $\Omega$ will be an (open or closed) $k$-connected
domain and we usually choose pole sets of minimal cardinality; that
is, having $k$ elements, one in each connected component of
$\widehat{\C}\setminus\overline{\Omega}$.

\goodbreak

\begin{definitions}
  Let $\Phi$ be a collection of functions mapping $\Omega$ into $\D$
  and analytic in neighborhoods of $\Omega$.  Fix a pole set $\Lambda$
  for $\Omega$.  We say that $\Phi$ is a
  \begin{enumerate}
  \item[$(i)$] \textit{uniform test collection over $\Omega$} if the
    implication \eqref{eq:tests} holds, where the constant $K$ depends
    only $\Omega$ and $\Phi$ (and not on $T$);
  \item[$(ii)$] \textit{quasi-uniform test collection over $\Omega$}
    if \eqref{eq:tests} holds, where $K$ depends on $\Omega$, $\Phi$
    and $S_\Lambda(T)$;
  \item[$(iii)$] \textit{non-uniform test collection over $\Omega$} if
    \eqref{eq:tests} holds, where $K$ can depend on $\Omega$, $\Phi$ and
    the operator~$T$;
  \item[$(iv)$] \textit{uniform strong test collection over $\Omega$}
    if \eqref{eq:tests-str} holds, where $K$ depends only on $\Omega$,
    $\Phi$ and $K'$ (but not on $T$);
  \item[$(v)$] \textit{quasi-uniform strong test collection over
      $\Omega$} if \eqref{eq:tests-str} holds, where $K$ depends on
    $\Omega$, $\Phi$, $K'$ and $S_\Lambda(T)$;
  \item[$(vi)$] \textit{non-uniform strong test collection over
      $\Omega$} if \eqref{eq:tests-str} holds, where $K$ depends on
    $\Omega$, $\Phi$, $K'$, and also may depend on $T$.
  \end{enumerate}
\end{definitions}

To summarize, there is the basic notion of a \emph{test collection},
which roughly means that whenever $\varphi(T)$ is a contraction for
every $\varphi$ in the collection, then $T$ has $\overline{\Omega}$ as
a $K$-spectral set.  To this, one can add the adjectives
\emph{uniform}, \emph{quasi-uniform} and \emph{non-uniform}, which
mean respectively that $K$ does not depend on $T$, that $K$ depends
only on $S_\Lambda(T)$, and that that $K$ may depend on $T$.  Finally,
the term \emph{strong} indicates that we can replace the condition
$\|\varphi(T)\| \leq 1$ by the weaker condition that $\overline{\D}$
is a complete $K'$-spectral set for $\phi(T)$ for all $\phi\in\Phi$.

An operator $R$ has $\overline{\D}$ as a complete $1$-spectral set if
and only if $R$ is a contraction.  In this case, $R$ has
$\overline{\D}$ as a complete $K'$-spectral set for all $K'>1$.
Therefore, each strong test collection is a test collection.

Also note that when $\Phi = \{\phi\}$ consists of a single element,
the \emph{strong} part comes for free, since if $\phi(T)$ has
$\overline{\D}$ as a complete $K$-spectral set for some $K$, then
there is some invertible operator $S$ such that
$S\varphi(T)S^{-1}=\varphi(STS^{-1})$ is a contraction, and so we can
reason with $STS^{-1}$ instead of $T$.

In most cases, $\Omega$ will be an open domain or the closure of an
open domain.  Given a domain $\Om$, the notions of a test collection
over $\Om$ and over $\ovl\Om$ might seem very similar, but as we will
see below, the condition that $\sigma(T) \subset \overline{\Omega}$,
as opposed to the stronger condition $\sigma(T) \subset \Omega$,
represents an additional technical challenge in some arguments.

Finally, the notion of a non-uniform test collection over an open set
$\Omega$ is trivial, since if $\sigma(T) \subset \Omega$, then
$\overline{\Omega}$ is a complete $K$-spectral set for $T$, where $K$
depends on $\Omega$ and $T$.  This was first proved for $\Omega=\D$ by
Rota~\cite{Rota}, and follows in general from the Herrero-Voiculescu
theorem (see \cite{Paulsen}*{Theorem 9.13}).

\subsection{Admissible function families}
\label{subsec:admiss-funct-famil}

Let us recall the definition of an admissible function
from~\cite{article1}.

\begin{definition}
  Let $\Omega \subset \C$ be a domain whose boundary is a disjoint
  finite union of piecewise analytic Jordan curves such that the
  interior angles of the ``corners'' of $\partial\Omega$ are in
  $(0,\pi]$.  We will say that an analytic function $\Phi =
  (\varphi_1,\ldots,\varphi_n): \overline{\Omega} \to \overline{\D}^n$
  is \emph{admissible} if $\varphi_k \in \A(\overline{\Omega})$, for
  $k = 1,\ldots,n$, and there is a collection of closed analytic arcs
  $\{J_k\}_{k=1}^n$ of $\partial\Omega$ and a constant $\alpha$,
  $0<\alpha\le 1$, such that the following conditions are satisfied:
  \begin{enumerate}[(a)]
  \item The arcs $J_k$ cover all $\partial\Omega$.
  \item $|\varphi_k| = 1$ in $J_k$, for $k = 1,\ldots,n$.
  \item For each $k = 1,\ldots,n$, there exists an open set $\Omega_k
    \supset \Omega$ such that the interior of $J_k$ relative to
    $\partial \Omega$ is contained in $\Omega_k$, $\varphi_k$ is
    defined in $\overline{\Omega}_k$, $\varphi_k \in
    A(\overline{\Omega}_k)$, and $\varphi_k'$ is of class H\"older
    $\alpha$ in $\Omega_k$, i.e.,
    \begin{equation*}
    |\varphi_k'(\zeta) - \varphi_k'(z)| \leq C |\zeta - z|^\alpha,
    \qquad \zeta,z \in \Omega_k.
    \end{equation*}
  \item If $z_0$ is an endpoint of $J_k$, then there exists an open
    sector $S_k(z_0)$ with vertex on $z_0$ and such that $S_k(z_0)
    \subset \Omega_k$ and $J_k \cap B(z_0,\varepsilon) \subset
    S_k(z_0) \cup \{z_0\}$, for some $\varepsilon > 0$.  Here,
    $B(z_0,\varepsilon)$ denotes the open disk of center $z_0$ and
    radius $\varepsilon$.  If $z_0$ is a common endpoint of both $J_k$
    and $J_l$, $k \neq l$, then we require that $(S_k(z_0) \cap
    S_l(z_0))\setminus\overline{\Omega}$ be nonempty.
  \item $|\varphi_k'| \geq C > 0$ in $J_k$, for $k = 1,\ldots,n$.
  \item For each $k = 1,\ldots,n$, $\varphi_k(\zeta) \neq
    \varphi_k(z)$ if $\zeta \in J_k$ and $z \in \overline{\Omega}$, $z
    \neq \zeta$.
  \end{enumerate}
\end{definition}

We recall from \cite{article1} that there is no loss of generality in
assuming in this definition that the arcs $J_k$ intersect only at
their endpoints.

\begin{figure}
  \begin{center}
    \includegraphics[width=10cm]{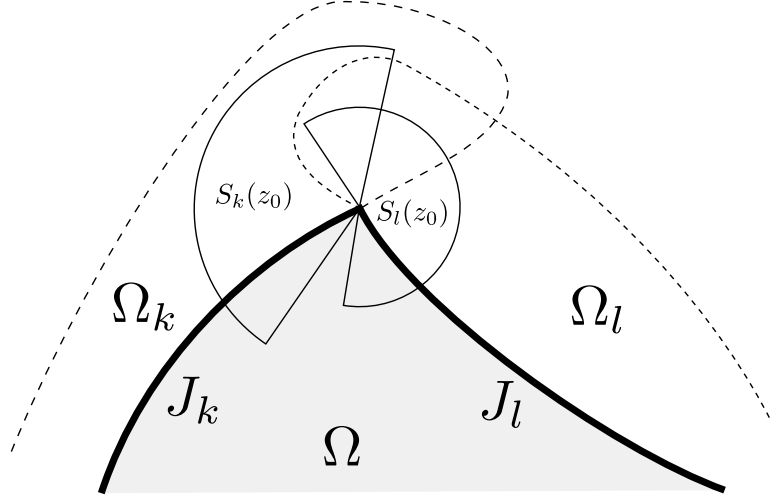}
    \caption{The geometric properties of an admissible function}
    \label{fig:0}
  \end{center}
\end{figure}

Given an admissible function $\Phi = (\varphi_1,\ldots,\varphi_n) :
\overline{\Omega}\to\overline{\D}^n$, we will denote the set of
functions $\{\varphi_1,\ldots,\varphi_n\}$ by the same letter $\Phi$.

\begin{theorem}
  \label{main}
  Assume that $\Phi : \overline{\Omega} \to \overline{\D}^n$ is
  admissible and analytic in an open neighborhood of
  $\overline{\Omega}$, where $\Omega $ is a Jordan domain.  Then
  $\Phi$ is a quasi-uniform strong test collection in
  $\overline{\Omega}$.  If, moreover, $\Phi$ is injective and $\Phi'$
  does not vanish on $\Omega$, then $\Phi$ is a uniform strong test
  collection over $\overline{\Omega}$.
\end{theorem}

This means that if $T \in \B(H)$ satisfies $\sigma(T) \subset
\overline{\Omega}$, $S_\Lambda(T))$ is an arbitrary fixed pole set for
$\Omega$, and $\varphi_k(T)$ have $\overline{\D}$ as a complete
$K'$-spectral set for $k = 1,\ldots,n$, then $T$ has
$\overline{\Omega}$ as a complete $K$-spectral set, with $K =
K(\Omega, \Phi, K', S_\Lambda(T))$.  If $\Phi$ is injective and
$\Phi'$ does not vanish on $\Omega$, then one can even choose $K$
independently of $T$.

\begin{theorem}
  \label{main2}
  Let $\Phi:\overline{\Omega}\to\overline{\D}^n$ be admissible and
  $\Lambda$ an arbitrary fixed pole set for $\Omega$.  Given $T \in
  \B(H)$, assume that there are operators $C_1,\ldots,C_n \in \B(H)$
  such that $\overline{\D}$ is a complete $K'$-spectral set for every
  $C_k$, $k = 1,\ldots,n$.  Furthermore, assume that whenever $f \in
  \Rat(\overline{\Omega})$ can be written as
  \begin{equation}
    \label{eq:C_k0}
    f(z) = \sum_{k=1}^n f_k(\varphi_k(z)), \qquad f_k \in
    A(\overline{\D}),
  \end{equation}
  we have
  \begin{equation}
    \label{eq:C_k}
    f(T) = \sum_{k=1}^n f_k(C_k).
  \end{equation}
  Then $\overline{\Omega}$ is a complete $K$-spectral set for $T$ with
  $K$ depending only on $\Omega$, $\Phi$ and $S_\Lambda(T)$.  If
  moreover, $\Phi$ is injective and $\Phi'$ does not vanish on
  $\Omega$, then one can choose $K$ independently of $T$.
\end{theorem}

\emph{A posteriori}, since $\overline{\Omega}$ is a complete
$K$-spectral for $T$, the operators $\varphi_k(T)$ are defined by the
$A(\overline{\Omega})$ calculus for $T$.  The hypotheses of the
theorem imply that $C_k = \varphi_k(T)$, so the operators $C_k$ are
uniquely defined.  However, \emph{a priori}, the operators
$\varphi_k(T)$ are not defined by any reasonable functional calculus,
so the theorem cannot be stated in terms of these operators.

If $\sigma(T) \subset \Omega$, then it is an easy consequence of the
Cauchy-Riesz functional calculus that $C_k = \varphi_k(T)$ satisfy the
hypotheses of this theorem.  Therefore, this proves the following
corollary.

\begin{corollary}
  \label{main-corollary}
  If $\Phi:\overline{\Omega}\to\overline{\D}^n$ is admissible, then
  $\Phi$ is a quasi-uniform strong test-collection over $\Omega$.  If
  moreover $\Phi$ is injective and $\Phi'$ does not vanish on
  $\Omega$, then $\Phi$ is a uniform strong test collection over
  $\Omega$.
\end{corollary}

\begin{remark}
  The main differences between Theorems~\ref{main} and \ref{main2} is
  that Theorem~\ref{main} assumes that $\Omega$ is simply connected
  and Theorem~\ref{main2} does not.  On the other hand,
  Theorem~\ref{main2} requires the existence of some operators $C_k$
  which behave in an informal sense like $\varphi_k(T)$ (the formal
  condition is that \eqref{eq:C_k0} implies \eqref{eq:C_k}).  As it
  will be clear from the proofs of these theorems, the case when
  $\sigma(T) \subset \Omega$ is easy to handle, while the case when
  $\sigma(T)$ contains part of the boundary of $\Omega$ presents some
  technical difficulties.  Theorems~\ref{main} and \ref{main2}
  represent two different ways of sorting out these difficulties.  In
  Theorem~\ref{main}, we will use the existence of a certain family
  $\{\psi_\eps\}_{0 \leq \eps \le \eps_0}$ of univalent functions on
  $\Om$ to pass from the operator $T$ to operators $\psi_\eps(T)$,
  whose spectra are contained over $\Omega$.  In Theorem~\ref{main2}
  we postulate some kind of functional calculus for $T$.  Ultimately,
  it would be desirable to extend Theorem~\ref{main} to multiply
  connected domains.
\end{remark}

We remark that it follows from the proofs of our theorems that similar
results hold if one replaces complete $K$-spectral sets by (not
necessarily complete) $K$-spectral sets.  For instance, in
Theorem~\ref{main2}, if $C_k$ have $\overline{\D}$ as a $K'$-spectral
set, then $\overline{\Omega}$ is $K$-spectral for $T$.

Theorems~\ref{Hav-Nersess-opers} and \ref{Putinar-Sandb-non-conv},
which were stated in the Introduction, can be reformulated in terms of
test collections.  Theorem~\ref{Hav-Nersess-opers} shows that if
$\varphi_k:\Omega_k \to \D$ are Riemann conformal maps, then
$\{\varphi_1,\ldots,\varphi_s\}$ is a uniform strong test collection
for $\overline{\Omega}$.  In Theorem~\ref{Putinar-Sandb-non-conv}, we
can put $\varphi_{k,\lambda}(z) = R(z-\mu_k(\lambda))^{-1}$.  Then
$\{\varphi_{k,\lambda} : k = 1,\ldots,N,\ \lambda\in\gamma_k\}$ is a
uniform test collection over $\overline{\Omega}$.

In Theorem~\ref{main}, it is easy to see that when $\Phi$ is not
injective or $\Phi'$ has zeros, then $\Phi$ can only be a non-uniform
strong test collection in $\overline{\Omega}$ (i.e., one cannot remove
the adjective ``non-uniform'').  For instance, if $\Phi(z_1) =
\Phi(z_2)$ for distinct points $z_1,z_2 \in \Omega$, then we can take
an operator $T$ acting on $\C^2$ and having $z_1$ and $z_2$ as
eigenvalues, with associated eigenvectors $v_1$ and $v_2$.  For every
$k$, we have $\varphi_k(T) = \varphi_k(z_1)I$, which is a contraction.
If the angle between $v_1$ and $v_2$ is very small, then $\|T\|$ will
be very large, so there is no constant $K$ independent of $T$ such
that $\overline{\Omega}$ is $K$-spectral for $T$.

Similarly, if $\Phi'(z_0) = 0$ for some $z_0 \in \Omega$, we can take
an operator $T$ such that $T \neq z_0I$ and $(T - z_0I)^2 = 0$.  For
every $n \geq 1$, we put $T_n = n(T - z_0I) + z_0I$.  Then it is easy
to check that for every $n$ and every $k$, we have
$\varphi_k(T_n)=\varphi_k(z_0)I$, which is a contraction.  However,
$\|T_n\| \to \infty$ as $n \to \infty$.  This implies that there is no
constant $K$ independent of $n$ such that $\overline{\Omega}$ is
$K$-spectral for $T_n$, for every $n$.

To illustrate the phenomenon described in the last paragraph,
construct a domain $\Omega$ and an admissible function
$\Phi:\overline{\Omega}\to \overline{\D}^3$ such that $\Phi'$ vanishes
at some point $z_0 \in \Omega$.  Choose a small $\varepsilon > 0$ and
put $z_1 = 0$, $z_2 = \varepsilon$, $z_3 = \varepsilon/2 +
i\sqrt{3}\varepsilon/2$, so that $z_1,z_2,z_3$ are the vertices of a
equilateral triangle of side length $\varepsilon$.  Let $z_0$ be the
center of this triangle.

Let $D_j$ be the disk of radius $1$ and center $z_j$.  We put $\Omega
= D_1 \cap D_2 \cap D_3$.  We can divide the boundary of $\Omega$ in
three arcs $J_k$ by putting $J_k = (\partial \Omega)\cap(\partial
D_k)$, for $k = 1,2,3$.  Since $\varepsilon$ is small, it is easy to
see that the length of each arc $J_k$ is close to $2\pi/3$.

Let $\varphi_1(z) = (z-z_0)^2/(1-\overline{z_0}z)^2$.  Then
$\varphi_1$ maps $D_1$ onto $\D$, and it maps $J_1$ bijectively onto
some arc of $\T$.  For $k = 2,3$, let $\eta_k$ be the
orientation-preserving rigid motion taking $z_k$ to $z_1$ and $J_k$ to
$J_1$ (so that it maps $\overline{D}_k$ onto $\overline{D}_1$).  Note
that $\eta_k(z_0) = z_0$.  We define $\varphi_k = \varphi_1 \circ
\eta_k$, for $k = 2,3$.  We see that $\varphi_k'(z_0) = 0$ for $k =
1,2,3$.  It is easy to check that $\Phi =
(\varphi_1,\varphi_2,\varphi_3)$ is admissible in $\overline{\Omega}$,
because, for every $k$, $\varphi_k$ is analytic on a neighborhood of
$\overline{\Omega}$ and takes $J_k$ bijectively onto some arc of $\T$.

It is worthwhile mentioning that the condition in the definition of an
admissible family of functions requiring the interior angle of a
corner of the domain $\Omega$ to be in $(0,\pi]$ can be relaxed in the
results stated above if one instead requires that the corner is not in
the spectrum of the operator $T$ under consideration.  This is seen by
altering $\Omega$, removing the intersection with a small enough disk
about the corner.  The complement of the disk will be a complete
spectral set for $T$ and the new corners created will satisfy the
condition that the interior angles are in $(0,\pi]$.  Since the disk
can be made arbitrarily small, it essentially has no effect on the
statements given above, other than that there is now a dependence on
the choice of $T$ through this additional requirement on the spectrum.

Part of the inspiration for our definition of test collections comes
from \cite{DritschelMcCullough}.  There, such a notion is defined
abstractly as a (possibly infinite) collection of complex valued
functions on a set with the property that at any given point in the
set, the supremum over the test functions evaluated at the point is
strictly less than $1$ and functions separate the points of the set.
In such cases as when the set $X$ is contained in $\mathbb C^n$, the
boundary of $X$ corresponds to points where some test function is
equal to $1$.  A test collection in this context is used to define the
dual notion of admissible kernels, and from these a normed function
algebra is constructed, with the functions in the test collection in
the unit ball of the algebra.  The realization theorem then states
that unital representations of the algebra which send the functions in
the test collection to strict contractions are (completely)
contractive.  In the case that the set where we define the test
collection is a bounded set $\Omega \subset \mathbb C$, this is
reminiscent of the test collection being a uniform test collection.
In the general setting of \cite{DritschelMcCullough}, the algebra
obtained may not be equal to $\A(\overline{\Omega})$, which is the
issue being addressed in this paper.

\section{Some examples of test collections from the literature}
\label{sec:known-criteria}

Here we interpret the known criteria for being a complete $K$-spectral
set in terms of our notion of a test collection and its variants.  For
a good recent review of different aspects of $K$-spectral sets and
complete $K$-spectral sets, the reader is referred
to~\cite{BadeaBech}.

\subsection{Intersection of disks}

A set $D\subset \C$ will be called a \emph{closed disk} in the Riemann
sphere $\widehat{\C}$ if it has of one of the following three forms:
\begin{equation*}
  \{z \in \widehat{\C} : |z - a| \leq r\},\quad
  \{z \in \widehat{\C} : |z - a| \geq r\},\quad \{z \in \widehat{\C} :
  \Re \alpha(z-a) \geq 0\},
\end{equation*}
i.e., it is either the interior of a disk in $\C$, the exterior of a
disk, or a half-plane.

\begin{theoremcite}[Badea, Beckermann, Crouzeix
  \cite{BadeaBeckermannCrouzeix}]
  \label{badea}
  Let $\{D_k\}_{k=1}^n$ be closed disks in $\widehat{\C}$ and
  $\{\varphi_k\}_{k=1}^n$ be fractional linear transformations taking
  $D_k$ onto $\overline{\D}$.  Then $\{\varphi_k\}_{k=1}^n$ is a
  uniform test collection for $\bigcap_{k=1}^n D_k$.
\end{theoremcite}

\subsection{Nice $n$-holed domains}

We say that an open bounded set $\Omega \subset \widehat{\C}$ is an
$n$-\emph{holed domain} if its boundary $\partial \Omega$ consists of
$n+1$ disjoint Jordan curves.  Given an $n$-holed domain $\Omega$, we
will denote by $\{U_k\}_{k=0}^n$ the connected components of
$\widehat{\C}\setminus \overline{\Omega}$, with $U_0$ the unbounded
component.  Let $X_k = \widehat{\C} \setminus U_k$.

\begin{theoremcite}[Douglas, Paulsen \cite{DouglasPaulsen}]
  \label{douglaspaulsen}
  Let $\Omega$ be an $n$-holed domain, and define $\{X_k\}_{k=0}^n$ as
  above.  Assume that each $X_k$ has an analytic boundary, so that there
  exist analytic homeomorphisms $\varphi_k : X_k \to \overline{\D}$,
  for $k = 0,\ldots,n$.  Then $\{\varphi_k\}_{k=0}^n$ is a uniform
  strong test collection in $\overline{\Omega}$.
\end{theoremcite}

This theorem can also be found in Paulsen's
book~\cite{Paulsen}*{Chapter 11}.

\subsection{Convex domains and the numerical range}

For $T \in \B(H)$, the \emph{numerical range} is defined as the set
\begin{equation*}
  W(T) = \{\langle Tx, x \rangle : \|x\| = 1\}.
\end{equation*}
It is well-know that this set is convex, and so its closure can be
written as the intersection of a (generally infinite) collection of
closed half planes $\{H_\alpha\}$.  Let $\varphi_\alpha$ be a linear
fractional transformation taking $H_\alpha$ onto $\overline{\D}$.  It
is easy to check that $W(T) \subset H_\alpha$ if and only if
$\|\varphi_\alpha(T)\| \leq 1$.  As we have already commented, it
follows from the arguments in \cite{DelyonDelyon} and
\cite{PutinarSandb} that every compact convex set containing $W(T)$ is
a complete $K$-spectral set for $T$.  This result can be rewritten in
terms of test collections as follows.

\begin{theoremcite}
  \label{numericalrange}
  Let $\Omega$ be a convex domain in $\C$ and let $\{H_\alpha\}$ be a
  collection of closed half-planes such that $\overline{\Omega} =
  \bigcap H_\alpha$.  Let $\varphi_\alpha$ be a fractional linear
  transform taking $H_\alpha$ onto $\overline{\D}$.  Then
  $\{\varphi_\alpha\}$ is a uniform test collection in
  $\overline{\Omega}$.
\end{theoremcite}

\begin{remark}
  If $\Omega$ is a smooth bounded convex set, we denote by $C_\Omega$
  the optimal constant $K$ such that $\overline{\Omega}$ is a
  (complete) $K$-spectral set for $T$ whenever $\overline{W(T)}
  \subset \Omega$.  The constant $Q = \sup_\Omega C_\Omega$ is know as
  Crouzeix constant.  Crouzeix has conjectured that $Q=2$.  The best
  result so far is $Q \leq 1 + \sqrt{2}$, as shown by Crouzeix and
  Palencia in their recent preprint~\cite{CrouzeixPalencia}.

  We also mention that in ~\cite{McCarthyPutinar}, a certain analogue
  of the Delyon-Delyon result ~\cite{DelyonDelyon} about a normal
  skew-dilation to the numerical range is given for a (possibly
  non-commuting) tuple of operators in the context of the symmetrized
  functional calculus.
\end{remark}

\subsection{$\rho$-contractions}

If $\rho > 0$, we say that an operator $T \in \B(H)$ is a
$\rho$-contraction if $T$ has an unitary $\rho$-dilation.  This is a
unitary operator $U$ acting on a larger Hilbert space $K \supset H$
and such that
\begin{equation*}
  T^n = \rho P_H U^n |H, \qquad n \geq 1.
\end{equation*}
Alternatively, one can ask that $\sigma(T) \subset \overline{\D}$ and
that the operator-valued Poisson kernel of $T$
\begin{equation*}
  K_{r,t}(T) = (I-re^{it}T^*)^{-1} + (I-re^{-it}T)^{-1} - I\qquad 0 <
  r < 1,\  t \in \R
\end{equation*}
satisfies
\begin{equation}
  \label{eq:operator-poisson-rho}
  K_{r,t}(T) + (\rho-1)I \geq 0, 0 < r < 1, t \in \R.
\end{equation}
The class of $\rho$ contractions becomes larger as $\rho$ increases,
as \eqref{eq:operator-poisson-rho} clearly shows, and $\rho=1$
corresponds to the usual contractions, while $\rho=2$ corresponds to
$W(T) \subseteq \mathbb D$.

If $1 < \rho < 2$, then $T$ being a $\rho$-contraction is also
equivalent to the condition that
\begin{equation}
  \label{eq:rho-1-2}
  \|\mu I - T\| \leq |\mu| + 1, \qquad \frac{\rho - 1}{2 - \rho} \leq
  |\mu| < \infty.
\end{equation}
(See, for instance, \cite{NagyFoias}*{Chapter I}.)  If $a \in \T$, we
denote by $D_a(\rho)$ the closed disk of radius $1 + (\rho - 1)/(2 -
\rho)$ whose boundary is tangent to $\T$ at $a$ and which contains
$\overline{\D}$.  Let $\varphi_{a,\rho}$ be a linear fractional
transformation taking $D_a(\rho)$ onto $\overline{\D}$.  Then
\eqref{eq:rho-1-2} is equivalent to the condition that
$\varphi_{a,\rho}(T)$ is a contraction for every $a \in \T$.

Similarly, if $\rho > 2$, $T$ is a $\rho$ contraction if and only if
\begin{equation}
  \label{eq:rho-2}
  \|(\mu I - T)^{-1}\| \leq \frac{1}{|\mu|-1}, \qquad 1 \leq |\mu|
  \leq \frac{\rho-1}{\rho-2}.
\end{equation}
For these values of $\rho$, denote by $D_a(\rho)$ the complement of
the open disk of radius $(\rho-1)/(\rho-2) - 1$, which is tangent to
$\T$ at $a \in \T$ and does not contain $\D$.  Let $\varphi_{a,\rho}$
be a linear fractional transformation which takes $D_a(\rho)$ onto
$\overline{\D}$.  Then \eqref{eq:rho-2} is equivalent to the condition
that $\varphi_{a,\rho}(T)$ is a contraction for every $a \in \T$.

For the case $\rho = 2$, let $D_a(2)$ be the closed half-plane which
is tangent to $\T$ at $a \in \T$ and which contains $\overline{\D}$
and let $\varphi_{a,2}$ be a fractional linear transformation taking
$D_a(2)$ onto $\overline{\D}$.  Then it follows from the above
comments regarding the numerical range that $T$ is a $2$-contraction
if and only if $\varphi_{a,2}(T)$ is a contraction for every $a \in
\T$.

It is also known \cite{NagyFoias} that every $\rho$-contraction is
similar to a contraction.  We summarize in terms of test collections
as follows.

\begin{theoremcite}
  \label{rho-contractions}
  For $\rho > 1$, let $\Phi_\rho = \{\varphi_{a,\rho}\}_{a \in \T}$,
  where $\varphi_{a,\rho}$ is defined as above.  Then $\Phi_\rho$ is a
  uniform test collection over $\overline{\D}$.
\end{theoremcite}

\subsection{Inner functions}

Recall that a Blaschke product is a function of the form
\begin{equation*}
B(z) = e^{i\theta}z^k \prod_{j=1}^N b_{\lambda_j}(z),
\end{equation*}
where
\begin{equation*}
b_\lambda(z) = \frac{\overline{\lambda}}{|\lambda|}\cdot \frac{\lambda
  - z}{1 - \overline{\lambda}z},
\end{equation*}
is a disk automorphism, $N$ may be either a finite number or $\infty$
(in which case its zeros $\lambda_j \in \D$ satisfy the Blaschke
condition $\sum_{j=1}^\infty (1-|\lambda_j|) < \infty$).  The Blaschke
product is called finite if $N$ is finite.

\begin{theoremcite}[Mascioni, \cite{Mascioni}]
  \label{mascioni}
  Let $\varphi$ be a finite Blaschke product.  Then the one element
  set $\{\varphi\}$ is a non-uniform strong test collection over
  $\overline{\D}$.
\end{theoremcite}

We cannot say that the one element set $\{\varphi\}$ is a uniform test
collection in $\overline{\D}$.  For example, take $\varphi(z) = z^2$,
which is a finite Blaschke product.  Then the operators $T_n$ on
$\C^2$ defined by the matrices $ T_n = {
  \begin{pmatrix}0 & n\\ 0 & 0\end{pmatrix} } $ satisfy $\varphi(T_n)
= 0$, but we have $\|T_n\| = n$.  Hence, $\overline{\D}$ can be a
$K$-spectral set for $T_n$ only if $K \geq n$.

In some of the theorems given above, the conclusion is that
some family of functions is a strong test collection, whereas in
others the conclusion is just that the family is a test collection.
Indeed, we do not know whether one can replace ``test collection'' by
``strong test collection'' in Theorems \ref{badea} and
\ref{numericalrange}.  The proofs of these theorems involve some kind
of operator valued Poisson kernel which turn out to be positive when
$\varphi(T)$ is a contraction, but they do not seem to work well if
$\varphi(T)$ simply has $\overline{\D}$ as a $K'$-spectral set.
Similarly, we do not know whether one can replace ``test collection''
by ``strong test collection'' in Theorem~\ref{rho-contractions}.

Theorem~\ref{mascioni} has been generalized by Stessin \cite{Stessin}
and Kazas and Kelley \cite{KazasKelley} to several classes of infinite
Blaschke products.  These generalizations give examples of test
functions on a set $\Omega$ which is neither an open domain nor its
closure.  We restate here Stessin's theorem in the language of test
collections.

\begin{theoremcite}[Stessin, \cite{Stessin}]
  Let $\varphi$ be a Blaschke product whose zeros
  $\{\lambda_j\}_{j=1}^\infty$ satisfy $\sum (1 - |\lambda_j|^2)^{1/2}
  < \infty$.  Let $\mathcal{P}$ be the set of poles of $\varphi$ and
  put $\Omega = \overline{\D}\setminus\overline{\mathcal{P}}$.  Then,
  the one element set $\{\varphi\}$ is a non-uniform strong test
  collection over~$\Omega$.
\end{theoremcite}

Another recent result that can be put into the terminology of test
collections is that of concerning lemniscates.

\begin{theoremcite}[Nevanlinna, \cite{Nevanlinna}]
  Let $p$ be a monic polynomial, $R > 0$, and denote by $\gamma_R$ the
  set $\{z \in \C : |p(z)|=R\}$.  Assume that no critical point of $p$
  lies on $\gamma_R$.  Let $\Omega = \{z \in \C : |p(z)| < R \}$ and
  $\varphi = p/R$.  Then the single-element set $\{\varphi\}$ is a
  non-uniform test collection over $\overline{\Omega}$.
\end{theoremcite}

\begin{remark}
  If $\Omega$ is a complete $K$-spectral set for an operator $T$, one
  can speak about constructing a concrete Sz.-Nagy-Foias like model of
  $T$ in $\Omega$.  In the simplest case, this model will be the
  compression of the multiplication operator $f\mapsto zf$ on the
  space $H^2(\Omega, U)\ominus \theta H^2(\Omega, Y)$, where $U, Y$
  are auxiliary Hilbert spaces and $\theta\in H^\infty(\Omega,
  L(Y,U))$ is an analogue of the characteristic function.  As it was
  shown in~\cite{Yakub}, there are important cases when the function
  $\theta$ can be calculated explicitly.  This is also true for some
  of the above examples.  If $T$ is a $\rho$-contraction, there is an
  explicit formula for its similarity to a contraction.  See, for
  instance,~\cite{OkuboAndo}.

  Such an explicit similarity transform is also available when
  $\|B(T)\|\le 1$ for a finite Blaschke product $B$.  Indeed, let
  \begin{equation*}
    B(z)=\prod_{k=1} b_k(z),
  \end{equation*}
  where $b_k(z)=(z-\la_k)/(1-\bar\la_k z)$ are Blaschke factors,
  $|\la_k|<1$.  The functions
  \begin{equation*}
    s_k(z)=\frac{(1-|\la_k|^2)^{1/2}}{1-\bar\la_k z}\,
    \prod_{j=1}^{k-1} b_j(z),\qquad k=1, \dots, n,
  \end{equation*}
  form an orthonormal basis of the model space $H^2\ominus B H^2$,
  whose reproducing kernel is $\big(1-\overline{B(w)}
  B(z)\big)/(1 - \bar w z)$.  For $z,w$ in a neighborhood of the
  closed unit disk, this gives
  \begin{equation*}
    1- \overline{B(w)} B(z)=(1 -\bar w z )\sum_{k=1}^n
    \overline{s_k(w)} s_k(z),
  \end{equation*}
  This then implies that for any $h\in H$,
  \begin{equation*}
    \sum_{k=1}^n \|s_k(T) h\|^2
    -
    \sum_{k=1}^n \|s_k(T) T h\|^2
    = \|h\|^2-\|B(T)h\|^2\ge 0,
  \end{equation*}
  and therefore $\|h\|^2_* : = \sum_{k=1}^n \|s_k(T) h\|^2$ defines a
  Hilbert space norm on $H$ for which $T$ is a contraction.  Since
  $s_1(T)$ is an invertible operator, this norm is equivalent to the
  original.

  As shown in \cite{Yakub}, there are ways of calculating the
  characteristic function $\theta$ of an operator $T$ explicitly
  without knowing an explicit form of the similarity transform
  converting $T$ into an operator for which $\Om$ is a complete
  spectral set with constant $1$.  As also explained in the paper, one
  obtains additional cases where explicit formulas are available by
  admitting a larger class of characteristic functions (which are then
  no longer unique).  Apart from these examples, we do not know either
  an explicit form of the similarity transform of the above type, nor
  explicit characteristic functions.
\end{remark}

\section{Proofs of Theorems~\ref{Hav-Nersess-opers}
and
\ref{Putinar-Sandb-non-conv}}
\label{proofs12}

In what follows, we will use the following lemma.  It is a special
case of a well known principle in the theory of $C^*$-algebras that
says that whenever the range of a linear map is commutative, the
complete boundedness of this map comes for free.

\begin{lemma}
  \label{lemma-cb}
  Let $T : A(\overline{\Omega}_1) \to A(\overline{\Omega}_2)$ be a
  bounded operator, and $\alpha \in A(\overline{\Omega})^*$ a bounded
  linear functional.  Then $T$ and $\alpha$ are completely bounded,
  and $\|T\|_\cb = \|T\|$ and $\|\alpha\|_\cb = \|\alpha\|$.
\end{lemma}

\begin{proof}[Proof]
  If $A$ and $B$ are $C^*$-algebras, with $B$ commutative, $S$ a (not
  necessarily closed) linear subspace of $A$ and $\phi : S \to B$ is a
  bounded linear map, then it is well known that $\phi$ is completely
  bounded and $\|\phi\|_\cb = \|\phi\|$ (see, for instance,
  \cite{Paulsen}*{Theorem 3.9} or \cite{Loebl}*{Lemma 1}).  The lemma
  then follows from the fact that $A(\overline{\Omega})$ is a subspace
  of the commutative $C^*$-algebra $C(\partial\Omega)$ and the norm in
  $A(\overline{\Omega})$ coincides with the norm that it inherits as a
  subspace of $C(\partial\Omega)$.
\end{proof}

\begin{proof}[Proof of Theorem~\ref{Hav-Nersess-opers}]
  We start by proving (ii) with $n=2$.

  Following the steps of \cite{HavinNersessian}*{Example 4.1}, we see
  that there are bounded operators $G_k : A(\overline{\Omega}) \to
  A(\overline{\Omega}_k)$, $k=1,2$, such that $f = G_1(f) + G_2(f)$,
  for every $f \in A(\overline{\Omega})$.  Indeed, although
  \cite{HavinNersessian}*{Example 4.1} is formulated for two simply
  connected domains whose boundaries intersect in only two points, the
  arguments used there are local at each point in
  $\partial\Omega_1\cap\partial\Omega_2$ and can be applied in our
  setting.  A key remark that we need to use here is that
  $X=\partial\Omega_1 \cap \partial\Omega_2$ is a finite set, a
  consequence of the transversality condition, which implies that
  every point $z \in X$ has an open neighborhood in $\widehat{\C}$
  which contains no other points from $X$.

  Take $f \in \Rat(\overline{\Omega}_1\cap\overline{\Omega}_2) \otimes
  M_s$, an $s\times s$ matrix-valued rational function with poles off
  $\overline{\Om}_1\cap \overline{\Om}_2$.  We first want to check
  that
  \begin{equation}
    \label{eq:calculus}
    f(T) = [(G_1\otimes \id_s)(f)](T) + [(G_2 \otimes \id_s)(f)](T).
  \end{equation}
  Note that the operators $[(G_k\otimes \id_s)(f)](T)$ are defined by
  the $A(\overline{\Omega}_k)$ calculus for $T$, which is well defined
  because each $\overline{\Omega}_k$ is a complete $K$-spectral set
  for $T$.

  The function $f$ can be decomposed as $f = f_1 + f_2$, with $f_j \in
  \Rat(\overline{\Omega}_j)\otimes M_s$, because any pole $a$ of $f$
  satisfies $a \in \widehat{\C}\setminus\overline{\Omega}_j$ for
  either $j=1$ or $j=2$.  Put $g_k = (G_k \otimes \id_s)(f)$, $k=1,2$.
  We have
  \begin{equation*}
  f_1 - g_1 = g_2 - f_2, \quad\text{in
  }\overline{\Omega}_1\cap\overline{\Omega}_2.
  \end{equation*}
  The left hand side of this equality belongs to
  $A(\overline{\Omega}_1)\otimes M_s$ and the right hand side belongs
  to $A(\overline{\Omega}_2)\otimes M_s$.  Thus this equation defines
  a function $h$ in $A(\overline{\Omega}_1 \cup \overline{\Omega}_2)
  \otimes M_s$ by $h = f_1 - g_1$ in $\overline{\Omega}_1$ and $h =
  g_2 - f_2$ in $\overline{\Omega}_2$.  Let $\{h_n\}_{n=1}^\infty
  \subset \Rat(\overline{\Omega}_1\cap\overline{\Omega}_2)\otimes M_s$
  be rational functions such that $h_n \to h$ uniformly in
  $\overline{\Omega}_1\cup\overline{\Omega}_2$.  Since $h_n \to f_1 -
  g_1 $ uniformly in $\overline{\Omega}_1$, we have that $h_n(T) \to
  f_1(T) - g_1(T) $ in operator norm.  On the other hand, since $h_n
  \to g_2 - f_2$ uniformly in $\overline{\Omega}_2$, we have that
  $h_n(T) \to g_2(T)-f_2(T)$ in operator norm.  Hence $f_1(T) - g_1(T)
  = g_2(T)-f_2(T)$.  This proves \eqref{eq:calculus}, because
  $f_1(T)+f_2(T) = f(T)$ by the rational functional calculus.

  Now we estimate
  \begin{equation*}
    \|f(T)\| \leq \sum_{k=1}^2 \|[(G_k\otimes \id_s)(f)](T)\| \leq K
    \left(\sum_{k=1}^2 \|G_k\|_\cb\right) \|f\|_{A(\overline{\Omega})
      \otimes M_s}.
  \end{equation*}
  By Lemma~\ref{lemma-cb}, $\overline{\Omega}$ is a complete
  $K'$-spectral set for $T$, with $K' = K(\|G_1\| + \|G_2\|)$.

  Now suppose that $n > 2$ and that the sets
  $\Omega_1,\ldots,\Omega_n$ satisfy the hypotheses of the theorem.
  Then the transversality conditions imply that
  $\Omega_1\cap\Omega_2,\Omega_3,\ldots,\Omega_n$ also satisfy the
  hypotheses of the theorem.  This enables us to apply induction in
  $n$.

  The proof of (i) is by the same argument, and indeed is somewhat
  simpler, since one only has to deal with scalar analytic functions
  and there is no need to invoke Lemma~\ref{lemma-cb}.
\end{proof}

\begin{remark}
  The proof relies essentially on the Havin-Nersessian separation of
  singularities: given some domains $\Om$, $\Om_1$, $\Om_2$ with good
  geometry such that $\Om=\Om_1\cap\Om_2$, any function $f$ in
  $H^\infty(\Om)$ admits a decomposition $f=f_1+f_2$, $f_j=G_j(f)\in
  H^\infty(\Om_j)$.  However, there are cases when the
  Havin-Nersessian separation fails, but nevertheless the following
  assertion is still true: if $\overline\Om_j$ is a ($1$-)spectral set
  for $T$, then $\overline\Om$ is a complete $K$-spectral set for $T$.
  This holds, for instance, if $\Om_1$ and $\Om_2$ are open
  half-planes such that $\Om_1\cup\Om_2=\C$.  In this case, $\Om_j$
  are simply connected, and so they are spectral sets for $T$ if and
  only if they are complete spectral sets.

  This assertion follows, for instance, from~\cite{Crouzeix}.
  However, there is no Havin-Nersessian separation in this case.  We
  reproduce arguments similar to those in
  \cite{HavinNersessian}*{Example 2.1} for the convenience of the
  reader.  By applying a linear map, we can assume that $\Om_1=\{\Re z
  <1\}$ and $\Om_2=\{\Re z >-1\}$.  Then the function
  $f(z)=\log((z+2)/(z-2))$ is in $H^\infty(\Om_1\cap\Om_2)$, but
  cannot be represented as $f=f_1+f_2$, $f_j\in H^\infty(\Om_j)$.
  Indeed, an easy application of a variant of the Liouville theorem
  shows that any representation $f = f_1 + f_2$ with analytic
  functions $f_j$ satisfying, say, $|f_j(z)|\le C(|z|+1)^{1/2}$ should
  satisfy $f_1(z)=S-\log(z-2)$, $f_2(z)=-S+\log(z+2)$, where $S$ is a
  constant, and in no case $f_j$ are bounded in $\Omega_j$.

  It is easy to modify this example to likewise produce bounded simply
  connected domains $\Om_1$ and $\Om_2$ with the same properties.
\end{remark}

\begin{question}
Consider Theorem~\ref{Hav-Nersess-opers} for the case of
two bounded and convex domains $\Om_1$ and $\Om_2$.
Can the constant $K'$ be
chosen to depend only on $K$ and not on the geometry of
$\Om_1$ and $\Om_2$?  Several modifications of this question
are possible, for instance, we can pose it for two general
Jordan domains or for domains, the boundaries of which have bounded
curvature.  Even when $K=1$ the questions are still of interest.
\end{question}

Now we pass to the proof of Theorem~\ref{Putinar-Sandb-non-conv}.  We
need some preliminaries.  Recall that we say that $D\subset
\widehat\C$ is a closed disk in $\widehat{\C}$ if it has of one of the
following three forms:
\begin{equation*}
\{z \in \widehat{\C} : |z - a| \leq r\},\quad
\{z \in \widehat{\C} : |z - a| \geq r\},\quad \{z \in \widehat{\C} :
\Re \alpha(z-a) \geq 0\},
\end{equation*}
i.e., it is either the interior of a disk in $\C$, the exterior of a
disk, or a half-plane.  We will refer to disks $\{z \in \widehat{\C} :
|z - a| \leq r\}$ as ``genuine'' disks.  Next, suppose $T$ is a
Hilbert space operator and $D$ has one of the above three forms.  We
say that $D$ is \textit{a good disk for $T$} if the following
condition holds (depending on the case):
\begin{itemize}
\item If $D=\{z \in \widehat{\C} : |z - a| \leq r\}$, we require that
  $\|T-a\|\le r$;
\item If $D=\{z \in \widehat{\C} : |z - a| \geq r\}$, we require that
  $a\notin \si(T)$ and $\|(T-a)^{-1}\|\le r^{-1}$;
\item If $D=\{z \in \widehat{\C} : \Re \alpha (z-a) \geq 0\}$, we
  require that $\Re \big(\al(T-a)\big)\geq 0$.
\end{itemize}

\begin{lemma}
  \label{lem-disks}
  Let $T$ be a Hilbert space operator.
  \begin{enumerate}
  \item[$(i)$] Suppose $\psi: \C\to\C$ is a M\"obius map, and that
    $\si(T)$ does not contain the pole of $\psi$, so that the operator
    $\psi(T)$ is bounded.  Then, given a closed disk $D$ in the
    Riemann sphere, $D$ is good for $T$ if and only if its image
    $\psi(D)$ is good for $\psi(T)$.
  \item[$(ii)$] Whenever $D_1\subset D_2$ are two Riemann sphere disks
    such that $D_1$ is good for $T$, the disk $D_2$ is also good for
    $T$.
  \end{enumerate}
\end{lemma}

\begin{proof}
  We will show that $D$ is good for $T$ if and only if $\sigma(T)
  \subset D$ and $\psi(T)$ is a contraction, where $\psi$ is a
  M\"obius transform taking $D$ onto $\overline{\D}$.  Part~$(i)$
  clearly follows from this property.  First note that if $\varphi$ is
  a disk automorphism, then $T$ is a contraction if and only if
  $\varphi(T)$ is a contraction.  Since every two M\"obius maps taking
  $D$ onto $\overline{\D}$ differ by composition on the right with a
  disk automorphism, we see that $\psi(T)$ is a contraction for every
  M\"obius map $\psi$ taking $D$ onto $\overline{\D}$ if and only if
  $\psi(T)$ is a contraction for some particular choice of such
  M\"obius map.

  Now we examine the three kinds of disks separately.  If $D = \{z :
  |z-a| \leq r\}$, then we can take $\psi(z) = (z-a)/r$ as a M\"obius
  map taking $D$ onto $\overline{\D}$.  The disk $D$ is good for $T$
  precisely when
  $\|\psi(T)\| \leq 1$.
  Similarly,
  a disk $D   = \{z : |z-a| \geq r\}$ is good for $T$ if and only
  if $\psi(T)$ is well-defined and is a contraction, where now
  we put $\psi(z) = r/(z-a)$, which is a
  M\"obius map taking $D$ onto $\overline{\D}$.

  In the last case, when $D = \{z : \Re \alpha(z-\alpha) \geq 0\}$ is a
  half-plane, $D$ is good for $T$ if and only if $\C_+ = \{z : \Re z
  \geq 0\}$ is good for $\alpha(T-a)$.  Hence it suffices to consider
  only the case when $D = \C_+$.  Using the
  standard fact that $\Re T \geq 0$ if and only if
  \begin{equation*}
    \|(I+T)x\|^2 \geq \|(I-T)x\|^2, \quad\forall x,
  \end{equation*}
  we see that $\Re T \geq 0$ if and only if $\psi(T)$ is a
  contraction, where now $\psi$ is a
  M\"obius map that takes $\C_+$ onto $\D$, given by
  $\psi(z) = (1-z)/(1+z)$.

  To prove~$(ii)$, we can use~$(i)$ to reduce first to the case when
  $D_1 = \overline{\D}$.  In this case, $T$ is a contraction.  Let
  $\psi$ be a M\"obius transform taking $D_2$ onto $\overline{\D}$.
  Then $|\psi| \leq 1$ in $\overline{\D}$, so $\psi(T)$ is a
  contraction by von Neumann's inequality.  It follows that $D_2$ is
  good for $T$, since $\sigma(T) \subset D_1 \subset D_2$.
\end{proof}

\begin{proof}[Proof of Theorem~\ref{Putinar-Sandb-non-conv}]
  By Condition~\ref{condA}, there are closed arcs
  $\gamma_1,\ldots,\gamma_N$ satisfying the hypotheses listed there.
  We are going to construct domains $\Om_1, \dots, \Om_N$, whose
  closures are complete $K$-spectral for $T$, with $\Om$ their
  intersection.  Then we will apply Theorem~\ref{Hav-Nersess-opers} to
  deduce that $\overline{\Om}$ is also complete $K'$-spectral for $T$,
  for some~$K'$.

  Fix $k \in \{ 1,\ldots,N \}$, and choose some point $z_k \in
  \bigcap_{\lambda\in\gamma_k} B(\mu_k(\lambda),R_k)$.  Put
  $\varphi_k(z) = (z-z_k)^{-1}$.  Now take some $\la\in \gamma_k$.
  Since $z_k\in B(\mu_k(\la), R_k)$, it follows that the closed disk
  \begin{equation*}
    D_\la^k=\phi_k\big(\C\sm B(\mu_k(\la), R_k)\big)
  \end{equation*}
  is genuine.  Since $\la \in\prt B(\mu_k(\la), R_k)$, we have
  $\phi_k(\la)\in\prt D_\la^k$.  Let $\ell_\la^k$ be the straight line
  tangent to $\partial D_\lambda^k$ at $\phi_k(\la)$ and let
  $\Pi_\la^k$ be the closed half plane bordered by $\ell_\la^k$ that
  contains $D_\la^k$.  Consider the (possibly unbounded) closed convex
  sets
  \begin{equation*}
    G_k=\bigcap_{\la\in \gamma_k} \Pi_\la^k.
  \end{equation*}
  Since $\Om\subset \widehat \C\sm B(\mu_k(\la),R_k)$, we have $
  \phi_k(\Om)\subset D_\la^k\subset \Pi_\la^k$, for any $\la\in
  \gamma_k$.  Therefore $\phi_k(\Om)\subset G_k$.  By
  Lemma~\ref{lem-disks}, the disk $D_\la^k$ and the half plane
  $\Pi_\la^k$ are good for $\phi_k(T)$.  It follows that
  $W(\varphi_k(T)) \subset G_k$.  By the Delyon-Delyon
  theorem~\cite{DelyonDelyon}, $G_k$ is a complete $K$-spectral set
  for $\phi_k(T)$ (see the comments in the \nameref{DelyonDelyon}).

  Next, we consider the Jordan domains
  $\Om_k=\operatorname{int}(\phi_k^{-1}(G_k))$ in the Riemann sphere
  $\widehat\C$.  Each $\Om_k$ contains $\Om$, and its closure is a
  complete $K$-spectral for $T$.  By construction, $\phi_k(\gamma_k)
  \subset \partial G_k$.  Hence, $\gamma_k \subset \partial \Om_k$.
  We wish to apply Theorem~\ref{Hav-Nersess-opers} to the intersection
  of the sets $\Om_k$, $k=1,\ldots,N$.  It may happen however that the
  boundaries of these sets do not intersect transversally.
  Nevertheless, it is possible to choose larger Jordan domains
  $\widetilde{\Om}_k \supset \Om_k$ whose boundaries do intersect
  transversally, and such that $\gamma_k \subset
  \partial \widetilde{\Om}_k$.

  To prove this, it suffices to choose the sets $\widetilde{\Om}_k$ in
  such a way that they intersect transversally at the endpoints of the
  arcs $\gamma_k$, as it is otherwise easy to ensure transversality at
  any other intersection points.  So suppose $\lambda$ is a common
  endpoint of two arcs $\gamma_k$ and $\gamma_l$.  By construction,
  the open disk $\Delta_k = \varphi_k^{-1}(\widehat{\C}\setminus
  \Pi_\lambda^k)$ has the point $\lambda$ on its boundary and does not
  intersect $\Om_k$, and similarly for the disk $\Delta_l$.
  Therefore, there is an open circular sector $S$ with vertex
  $\lambda$ that does not intersect
  $\overline{\Om}_k\cup\overline{\Om}_l$.  Since $\gamma_k$ and
  $\gamma_l$ intersect transversally, we can find disjoint open
  circular sectors $S_k^+$ and $S_l^+$ which are also disjoint with
  $S$ and such that $\gamma_k\cap B(\lambda,\varepsilon) \subset S_k^+
  \cup \{\lambda\}$ and $\gamma_l\cap B(\lambda,\varepsilon) \subset
  S_l^+ \cup \{\lambda\}$ for some $\varepsilon > 0$.  Now observe
  that we can choose disjoint open circular sectors $S_k^-$ and
  $S_l^-$ which are also disjoint from $S,S_k^+,S_l^+$, and then the
  larger sets $\widetilde{\Omega}_k \supset \Omega_k$ and
  $\widetilde{\Omega}_l \supset \Omega_l$ to satisfy
  $(\partial\widetilde{\Omega}_k \setminus \gamma_k)\cap
  B(\lambda,\varepsilon) \subset S_k^-$ and
  $(\partial\widetilde{\Omega}_l \setminus \gamma_l)\cap
  B(\lambda,\varepsilon) \subset S_l^-$.  Consequently,
  $\widetilde{\Omega}_k$ and $\widetilde{\Omega}_l$ intersect
  transversally at $\lambda$.

  Put
  \begin{equation*}
    \wt\Om=\widetilde{\Om}_1\cap \dots \cap \widetilde{\Om}_N.
  \end{equation*}
  By Theorem~\ref{Hav-Nersess-opers}, the closure of $\wt\Om$ is a
  complete $K'$-spectral set for $T$, for some $K'$.  By construction,
  each point $\lambda$ of $\partial\Om$ has a neighborhood $B(\lambda,
  \varepsilon)$ such that $B(\lambda, \varepsilon)\cap \partial\Om=
  B(\lambda, \varepsilon)\cap
  \partial\widetilde\Om$.  Therefore $\wt\Om\sm \Om$ is at a positive
  distance from $\Om$.  Since $\Om\subset \wt\Om$ and
  $\si(T)\subset\overline{\Om}$, it follows that $\overline{\Om}$ also
  is a complete $K''$-spectral set for $T$.
\end{proof}

We recall that a Hilbert space operator $T$ is \emph{hyponormal} if
$T^*T \geq T T^*$.  In this case, the equality
$\|(T-\la)^{-1}\|=1/\operatorname{dist}(\la, \si(T))$ holds for all
$\la\notin \si(T)$; see, for instance the book \cite{MartPutinar} by
Martin and Putinar, Proposition~1.2.  Consequently, we get the
following corollary to Theorem~\ref{Putinar-Sandb-non-conv}.

\begin{corollary}
  \label{cor-hypo}
  Let $T$ be hyponormal and let $\Om\subset\widehat{\C}$ be an open
  set satisfying the hypotheses of
  Theorem~\ref{Putinar-Sandb-non-conv} $($in particular, the exterior
  disc condition$)$ and such that $\sigma(T)\subset\overline{\Omega}$.
  Then $\overline\Om$ is a complete $K$-spectral set for $T$.
\end{corollary}

So in other words, in this situation, $T$ can be dilated to an
operator $S$ which is similar to a normal operator and satisfies
$\si(S)\subset \prt\Om$.  It is interesting to compare
Corollary~\ref{cor-hypo} with Putinar's result \cite{Putinar-subsc}
that every hyponormal operator $T$ is subscalar and can in fact can be
represented as a restriction of a scalar operator $L$ of order $2$ (in
the sense of Colojoara-Foias) to an invariant subspace.  Thus, if
$T=L|H$, where $H$ is an invariant subspace of $L\in B(K)$, then $L$
is a dilation of $T$ of a special kind.  On the other hand, the
spectrum of a scalar operator $L$, as constructed by Putinar, contains
a neighborhood of $\si(T)$.  By contrast, in Corollary~\ref{cor-hypo},
if $\si(T)$ is a closed Jordan domain satisfying the exterior disk
condition, the dilation $S$ of $T$ is a scalar operator of order $0$
and its spectrum is contained in the spectrum of $T$ (and even in its
boundary).

Generally speaking, the conditions of Corollary~\ref{cor-hypo} do not
imply that $\overline\Om$ is a ($1$-)spectral set for $T$; this is
seen from any of the examples by Wadhwa \cite{wadhwa} and Hartman
\cite{Hartman}, where one can put $\overline{\Om}=\sigma(T)$ (it is an
annulus for the Hartman's example and a disjoint union of an annulus
and a disc for Wadhwa's example).  On the other hand, consider the
hyponormal operator from Clancey's example \cite{clancey}; let us call
it $B$.  Its spectrum is a compact subset of $\C$ of positive area,
whose interior is empty.  It is proved in \cite{clancey} that $\si(B)$
is not a $1$-spectral set for $B$.  A modification of Clancey's
arguments also shows that it is not even $K$-spectral for any $K$.
Indeed, by applying \cite{Paulsen}*{Exercise 9.11}, one gets that if
$\si(B)$ were $K$-spectral for $B$, then $B$ would be similar to a
normal operator.  Since $B$ is hyponormal,
\cite{Stampfli-Wadhwa}*{Corollary 1} would then give that $B$ is
normal, which is not true.  So the equality
$\|(T-\la)^{-1}\|=1/\operatorname{dist}\big(\la, \si(T)\big)$
($\la\notin\si(T)$) in general does not imply that $\si(T)$ is a
$K$-spectral set for $T$.

We also refer to \cite{Putinar-BanCenter38}*{Theorem 4} for a result
on subscalarity of operators with a power-like estimate for the
resolvent.

\section{Admissible functions and generators of $A(\overline{\Omega})$}
\label{sec:article1}

In this section we state the results from \cite{article1} needed in
this article.

\begin{theoremarticle}[Theorem 1.5 in \cite{article1}]
  \label{article1-thm1}
  Let $\Phi:\overline{\Omega}\to\overline{\D}^n$ be admissible.  Then
  there exist bounded linear operators $F_k : A(\overline{\Omega}) \to
  A(\overline{\D})$, $k = 1,\ldots,n$ such that the operator in
  $A(\overline{\Omega})$ defined by
  \begin{equation*}
    f \mapsto f - \sum_{k=1} F_k(f) \circ \varphi_k,\qquad f \in
    A(\overline{\Omega}),
  \end{equation*}
  is compact.
\end{theoremarticle}

If $\Phi = (\varphi_1,\ldots,\varphi_n) : \overline{\Omega} \to
\overline{\D}^n$, we define the algebra $\mathcal{A}_\Phi$ to be the
(not necessarily closed) subalgebra of $A(\overline{\Omega})$
generated by the functions of the form $f \circ \varphi_k$, with $f
\in A(\overline{\D})$ and $k = 1,\ldots,n$.  Explicitly,
\begin{equation*}
\mathcal{A}_\Phi = \bigg\{ \sum_{j=1}^{N}
f_{j,1}(\varphi_1(z))\cdot\,\ldots\,\cdot f_{j,n}(\varphi_n(z)) \ :\
f_{j,k} \in A(\overline{\D}),\ N \in \N \bigg\}.
\end{equation*}

\begin{theoremarticle}[Theorem 1.1 in \cite{article1}]
  \label{article1-thm2}
  Let $\Phi:\overline{\Omega}\to\overline{\D}^n$ be admissible and
  injective, and such that $\Phi'$ does not vanish on $\Omega$.  Then
  $\mathcal{A}_\Phi = A(\overline{\Omega})$.
\end{theoremarticle}

\begin{lemmaarticle}
  \label{article1-lemma}[Lemma 8.1 in \cite{article1}]
  Let
  $\Phi_\varepsilon= \big(\varphi_1^\varepsilon, \dots,
  \varphi_n^\varepsilon\big): \overline{\Omega}\to\overline{\D}^n$,
  $0\leq\varepsilon\leq\varepsilon_0$ be a collection of functions.
  Assume that $\Psi_\varepsilon$ is admissible for every
  $\varepsilon$, and, moreover, that one can choose sets $\Omega_k$ in
  the definition of an admissible collection that do not depend on
  $\varepsilon$.
  Assume that $\varphi_k^\varepsilon \in \Cont^{1+\alpha}(\Omega_k)$,
  with $0 < \alpha < 1$, and that the mapping $\varepsilon \mapsto
  \varphi_k^\varepsilon$ is continuous from $[0,\varepsilon_0]$ to
  $\Cont^{1+\alpha}(\Omega_k)$.

  Then for $0\leq\varepsilon\leq \varepsilon_0$, there exist bounded
  linear operators $F_k^\varepsilon : A(\overline{\Omega}) \to
  A(\overline{\D})$, such that if
  \begin{equation*}
  L_\varepsilon(f) = \sum F_k^\varepsilon(f)\circ\varphi_k^\varepsilon
  \end{equation*}
  then $L_\varepsilon - I$ is a compact operator on
  $A(\overline{\Omega})$ for all $\varepsilon$, the mapping
  $\varepsilon \mapsto L_\varepsilon$ is continuous in the norm
  topology, and $\|F_k^\varepsilon\| \leq C$ for $k = 1,\ldots,n$,
  where $C$ is some constant independent of $k$ and~$\varepsilon$.
\end{lemmaarticle}

\section{Auxiliary lemmas}
\label{lemmas}

In this section we state and prove the lemmas that are needed in the
proof of the Theorems~\ref{main} and~\ref{main2}.

\begin{lemma}
  \label{lemma-complement-polynomials}
  Let $X$ be a closed subspace of finite codimension $r$ in a Banach
  space $V$ and $Y$ a $($not necessarily closed$)$ subspace of $V$
  such that $X + Y = V$.  Then there exist vectors $g_1,\ldots,g_r \in
  Y$ such that $Z = \operatorname{span}\{g_1,\ldots,g_r\}$ is a
  complement of $X$; that is, $V = X \dotplus Z$, and there are
  functionals $\alpha_1,\ldots,\alpha_r \in V^*$ such that
  \begin{equation*}
    G(f) \overset{\text{def}}{=} f - \sum \alpha_k(f)g_k
  \end{equation*}
  is the projection of $V$ onto $X$ parallel to $Z$.
\end{lemma}

\begin{proof}
  Let $\pi : V \to V/X$ be the natural projection onto the quotient.
  Then $\pi(Y) = \pi(X+Y) = V/X$.  We can therefore choose vectors
  $g_1,\ldots,g_r \in Y$ such that $\{\pi(g_1),\ldots,\pi(g_r)\}$ is a
  basis of $V/X$.  It follows that $Z =
  \operatorname{span}\{g_1,\ldots,g_r\}$ is a complement of $X$ in
  $V$.  The existence of the functionals $\alpha_1,\ldots,\alpha_r$ is
  now clear.
\end{proof}

Note that, since $X+Y$ is always closed, the hypotheses of the lemma
in particular hold in the case when $Y$ is a dense subspace of $V$.

The next lemma roughly says that to prove von Neumann's inequality
with a constant, it is enough to prove it only for rational functions
in a space of finite codimension.

\begin{lemma}
  \label{lemma-finite-codim}
  Let $T \in \B(H)$ be such that for all $s \geq 1$,
  \begin{equation}
    \label{eq:dm-**}
    \|f(T)\| \leq C \|f\|_{A(\overline{\Omega})\otimes M_s},\qquad
    \forall f \in (X \cap \Rat(\overline{\Omega}))\otimes M_s,
  \end{equation}
  where $X$ is some closed subspace of finite codimension in
  $A(\overline{\Omega})$.  Then $\overline{\Omega}$ is a complete
  $K$-spectral set for $T$, where $K$ depends only on $X$, $C$ and
  $S_\Lambda(T)$, where $\Lambda$ is an arbitrary pole set for
  $\Omega$.
\end{lemma}

\begin{proof}
  Fix a pole set $\Lambda$ for $\Omega$.  Denote by $\Rat_\Lambda$ the
  set of rational functions with poles in $\Lambda$.  Note that
  $\Rat_\Lambda$ is dense in $A(\overline{\Omega})$.  Hence, we apply
  Lemma~\ref{lemma-complement-polynomials} with $V =
  A(\overline{\Omega})$, $Y = \Rat(\overline{\Omega})$ to obtain
  functions $g_1,\ldots,g_r \in \Rat_\Lambda$, functionals
  $\alpha_1,\ldots,\alpha_r\in A(\overline{\Omega})^*$, and an
  operator $G:A(\overline{\Omega}) \to A(\overline{\Omega})$ as in the
  statement of that lemma.

  We can write
  \begin{equation*}
    g_k(z) = c_0 + \sum_{\lambda \in \Lambda}\sum_{j=1}^N
    c_{\lambda,j,k} p_\lambda(z)^j
  \end{equation*}
  for suitable coefficients $c_{\lambda,j,k}$.  (Recall that
  $p_\lambda(z) = (z-\lambda)^{-1}$ for $\lambda\neq\infty$, and
  $p_\infty(z) = z$.)  This shows that for $k=1\ldots,r$, $\|g_k(T)\|
  \leq K'$, where $K'$ is a constant depending only on
  $g_1,\ldots,g_r$ and $S_\Lambda(T)$, but not on $T$.

  Let $f \in \Rat(\overline{\Omega}) \otimes M_s$.  By
  Lemma~\ref{lemma-cb}, $G$ and $\alpha_1,\ldots,\alpha_r$ are
  completely bounded, so by~\eqref{eq:dm-**},
  \begin{equation*}
    \begin{split}
      \|f(T)\| &= \Big\|[(G \otimes \id_s)(f)](T) + \sum_k g_k(T)
      \otimes [(\alpha_k \otimes \id_s)(f)]\Big\|
      \\
      &\leq \|[(G \otimes \id_s)(f)](T)\| + \sum_{k=1}^n \|g_k(T)
      \otimes [(\alpha_k \otimes \id_s)(f)]\|
      \\
      &\leq C\|(G \otimes \id_s)(f)\|_{A(\overline{\Omega}) \otimes
        M_s} + \sum_{k=1}^r K'\|\alpha_k \otimes \id_s\|\cdot
      \|f\|_{A(\overline{\Omega}) \otimes M_s}
      \\
      &\leq C\|G\|_\cb\|f\|_{A(\overline{\Omega})\otimes M_s} +
      \sum_{k=1}^r K'\|\alpha_k\|_\cb \|f\|_{A(\overline{\Omega})
        \otimes M_s},
    \end{split}
  \end{equation*}
  and the result follows.
\end{proof}

\begin{definition}
  Given a domain $\Omega \subset \C$, a \emph{shrinking} for $\Omega$
  is a collection $\{\psi_\eps\}_{0 \leq \eps \le\eps_0}$ of univalent
  analytic functions in some open set $U \supset \overline{\Omega}$,
  such that $\psi_0$ is the identity map on $U$,
  $\psi_\varepsilon(\overline{\Omega}) \subset \Omega$ for
  $\varepsilon > 0$, and the map $\varepsilon \mapsto
  \psi_\varepsilon$ is continuous in the topology of uniform
  convergence on compact subsets of~$U$.
\end{definition}

If $\Omega$ is star-shaped with respect to $a \in \C$, then it admits
a shrinking; namely, $\psi_\varepsilon(z) = (1-\varepsilon)(z-a) + a$.
The next lemma says that any admissible Jordan domain admits a
shrinking.

\begin{lemma}
  \label{lemma-shrinking}
  Let $\Om$ be a Jordan domain with piecewise $C^2$ smooth boundary
  composed of closed $C^2$ arcs $\{J_k\}_{k=1}^n$.  If the angles
  between these arcs are non-zero, then $\Omega$ admits a shrinking.
\end{lemma}

\begin{proof}
  Denote by $z_1, \dots, z_n\in \partial\Om$ the endpoints of the arcs
  $J_1, \dots, J_n$, so that $z_k$ is a common endpoint of $J_{k-1}$
  and $J_k$ (we assume that the numbering of these arcs is
  counterclockwise and cyclic modulo $n$).  The points $z_1, \dots,
  z_n$ will be referred to as the \emph{corners} of $\prt\Om$.  First
  we construct a function $\mu\in A(\overline\Om)$ such that $\mu\ne
  0$ on $\prt\Om$ and for any $z\in\overline\Om$, $\mu(z)$ points
  strictly inside $\Om$, by which we mean that there is $\si=\si(z)>0$
  such that the interval $[z,z+ \si \mu(z)]$ is contained in
  $\overline\Om$ and is not tangent to $\prt\Om$ at $z$.  If $z=z_k$,
  we require this interval to be non-tangential to both $J_{k-1}$ and
  $J_k$ at $z$.  We denote by $\rho(z)\in \C$ the unit inner normal
  vector to the boundary.  It is defined for points $z\in \prt\Om$
  which are not corners.

  Let $\eta: \ovl\Om\to\ovl\D$ be a Riemann conformal map, and put
  \begin{equation*}
    \nu(z)= - \frac {\eta(z)} {\eta'(z)}.
  \end{equation*}
  Then $\nu$ is continuous on $\ovl\Om\setminus \{z_1,\dots,z_n\}$;
  moreover, $\nu(z_0)$ points strictly inside $\Om$ for any non-corner
  point $z_0\in \prt\Om$ and, in fact, $\rho(z_0)= c(z_0) \nu(z_0)$
  for some $c(z_0)>0$.  Indeed, if $z(t)$ $(0\le t\le T)$ is a
  counterclockwise parametrization of $\prt\Om$, which is smooth at
  $z_0$, $z(t_0)=z_0$, $z'(t_0)=b$, then $\eta'(z_0)b = i c \eta(z_0)$
  for some $c>0$, so that $\rho(z_0)= ib = c\nu(z_0)$.  The function
  $\mu(z)$ will be, in a sense, a small correction of $\nu(z)$, which
  mostly affects neighborhoods of the corner points.

  Denote by $R_{z,\tht}=\{w\in \C: \arg(w-z)=\tht\}$ the ray starting
  at $z$ with angle $\tht$.  We assume that the rays $R_{z_k,
    \tht_k-\beta_k}$, $R_{z_k, \tht_k+\beta_k}$ are correspondingly
  tangent to $\partial\Omega$ at $z_k$ to the arcs $J_{k-1}$, $J_k$,
  where $0<\beta_k<\pi$, and $\tht_k\in [0,2\pi)$ is such that the ray
  $R_{z_k, \tht_k}$ points strictly inside $\Om$.  Theorem 3.9 in
  \cite{Pommerenke} implies that for $z\in \ovl\Om$,
  \begin{align}
    \label{asympt-Pomm}
    \eta(z) & =\eta(z_k) + u_k(z) (z-z_k)^{a_k}, \\
    \label{asympt-Pomm'}
    \eta'(z) & =v_k(z) a_k (z-z_k)^{a_k-1},
  \end{align}
  where $u_k(z)$, $v_k(z)$ have finite non-zero limits as $z\to z_k$,
  and $a_k=\frac \pi {2\beta_k} \in (\frac{1}{2},+\infty)$.  (We use the
  principal branch of the logarithm in the definition of powers.)  For
  small $\si>0$, put $\tau_{k,\si}:= z_k- \si e^{i\tht_k}
  \notin\ovl\Om$, and set
  \begin{equation*}
    \mu_\si(z)= \Pi_\si(z) \nu(z) = - \Pi_\si(z) \frac
    {\eta(z)}{\eta'(z)},
  \end{equation*}
  where
  \begin{equation*}
    \Pi_\si(z)= \prod_{k=1}^n \bigg(\frac {z-\tau_{k,\si}} {z-z_k}
    \bigg)^{1-a_k} \; .
  \end{equation*}
  Since the intervals $[z_k, \tau_{k,\si}]$ are outside $\Om$, the
  function $\Pi_\sigma$ is well-defined and analytic in $\Om$.

  We assert that for sufficiently small $\sigma > 0$,
  $\mu(z)=\mu_\si(z)$ will satisfy all the necessary requirements.  To
  begin with, it follows from \eqref{asympt-Pomm} and
  \eqref{asympt-Pomm'} that for any fixed (small) $\si>0$ and any $k$,
  $\mu_\si(z)$ has a finite non-zero limit as $z\to z_k$, $z\in
  \ovl\Om$.  Hence $\mu_\si$ continues to a function in $A(\ovl\Om)$
  such that $\mu_\si\ne0$ on $\prt\Om$.

  Fix some small positive $\de$ such that for all $k$, $2\de<\beta_k <
  2\pi - 2\de$.  Easy geometric arguments show that there is some
  $\si_0>0$ such that for any $k$, any $z\in J_{k-1}$ such that
  $|z-z_k|<\si_0$ and any $\si\in (0, \si_0)$, either $\arg
  \Pi_\si(z)\in (-\de, \beta_k- \frac \pi 2 + \de)$ if $\beta_k\in [\frac
  \pi 2, \pi)$, or $\arg \Pi_\si(z)\in (\beta_k- \frac \pi 2 + \de,
  \de)$ if $\beta_k\in (0,\frac \pi 2)$.  One has symmetric estimates
  for $\arg \Pi_\si(z)$ if $z\in J_k$, $|z-z_k|<\si_0$.  Since
  $\Pi_\si(z)\to 1$ uniformly on $\prt\Om\setminus \cup_k
  B_{\si_0}(z_k)$, it follows that for any $z\in \prt\Om$, $z\ne z_1,
  \dots, z_n$ when $\si\in (0, \si_0)$ is sufficiently small,
  \begin{equation*}
    -\frac \pi 2 + \de \le \arg \frac {\mu_\si(z)}{\rho(z)} = \arg
    \Pi_\si(z) \le \frac \pi 2-\de.
  \end{equation*}
  For such fixed $\si$, $\mu(z) := \mu_\si(z)$ satisfies all the
  requirements.

  By Mergelyan's theorem \cite{Rudin-book}, there is a sequence of
  polynomials $\mu_m$, such that $\mu_m\to \mu$ uniformly on
  $\ovl\Om$.  For a sufficiently large $m$, put $\tilde \mu(z) =
  \mu_m(z)$.  Then the polynomial $\tilde \mu$ also satisfies all the
  requirements on $\mu$.

  We assert that $\psi_\eps(z)= z + \eps \tilde \mu(z)$ defines a
  shrinking of $\Om$.  Indeed, for small $\eps>0$,
  $\psi_\eps(\prt\Om)$ is a Jordan curve, contained in $\Om$.  An
  application of the argument principle shows that for these values of
  $\eps$, $\psi_\eps$ maps $\Om$ univalently onto the interior of the
  curve $\psi_\eps(\prt\Om)$.  There exists a Jordan domain $\Om'$
  such that $\ovl{\Om}\subset \Om'$ and the boundary of $\Om'$
  consists of $C^2$ smooth arcs $J_1', \dots, J_n'$, which are close
  to the arcs $J_1, \dots, J_n$ in $C^1$ metric.  The domain $\Om'$
  can be chosen in such a way that $\tilde\mu\ne0$ on $\prt\Om'$ and
  $\tilde\mu(z)$ points strictly inside $\Om'$ for all $z\in
  \prt\Om'$.  By the same argument, the functions $\psi_\eps$ are
  univalent on $\Om'$ for $0\le \eps\le\eps_0$, where $\eps_0>0$.
  Therefore the family $\{\psi_\eps\}$ of functions, defined on
  $\Om'$, is a shrinking of $\Om$.
\end{proof}

The following lemma improves upon the results of
Lemma~\ref{lemma-finite-codim} by imposing certain constraints on
$\varphi_k(T)$ and $\mathcal{A}_\Phi$.

\begin{lemma}
  \label{lemma-finite-codim-Aphi}
  Let $\Phi \subset A(\overline{\Omega})$ be a collection of functions
  taking $\Omega$ into $\D$.  If, in addition to the hypotheses of
  Lemma~\ref{lemma-finite-codim}, we have that for every
  $\varphi\in\Phi$, $\overline{\D}$ is a (not necessarily complete)
  $K'$-spectral set for $\varphi(T)$, then for all $s \geq 1$,
  \begin{equation}
    \label{eq:this-lemma-*}
    \|f(T)\| \leq K\|f\|_{A(\overline{\Omega})\otimes M_s},
    \qquad
    \forall f \in (X + \mathcal{A}_\Phi) \otimes
    M_s,
  \end{equation}
  where $K$ depends only on $X$, $\Phi$, $C$ and $K'$, but not on $T$.
  In the case when $X + \mathcal{A}_\Phi = A(\overline{\Omega})$, then
  $\overline{\Omega}$ is a complete $K$-spectral set for $T$, with
  $K=K(X,\Phi,C,K')$.
\end{lemma}

Note that the operators $\varphi(T)$ and $f(T)$ are defined by the
$A(\overline{\Omega})$-functional calculus for $T$ because by
Lemma~\ref{lemma-finite-codim}, $\overline{\Omega}$ is a complete
$K$-spectral for $T$ for some $K$.  On the other hand, and in contrast
to the situation in most of this paper, here the complete
$K'$-spectrality of $\overline{\D}$ for $\varphi(T)$ is not needed ---
$K'$-spectrality suffices.  The reason for this is that all the
functions that appear in the proof of this lemma are scalar-valued
rather than matrix-valued.

\begin{proof}[Proof of Lemma~\ref{lemma-finite-codim-Aphi}]
  First we apply Lemma~\ref{lemma-complement-polynomials} with $V = X
  + \mathcal{A}_\Phi$, and $Y = \mathcal{A}_\Phi$ to obtain functions
  $g_1,\ldots,g_r \in \mathcal{A}_\Phi$, functionals
  $\alpha_1,\ldots,\alpha_r\in A(\overline{\Omega})^*$, and an
  operator $G$ as in the statement of that lemma.

  By Lemma~\ref{lemma-finite-codim}, $\overline{\Omega}$ is a complete
  $K$-spectral set for $T$ (with $K$ depending on $T$).  It follows
  that $T$ has a continuous $A(\overline{\Omega})$-functional
  calculus, and so the operators $g_k(T)$ are well defined.  Let us
  show that there is some constant $C'$ depending only on
  $g_1,\ldots,g_r$ (and not on $T$) such that $\|g_k(T)\| \leq C'$,
  for $k = 1,\ldots,r$.  Since $g_k \in \mathcal{A}_\Phi$, we can
  write
  \begin{equation*}
    g_k(z) = \sum_{j=1}^N f_{j,1}^k(\varphi_1(z))\cdot\,\cdots\,\cdot
    f_{j,n}^k(\varphi_n(z)),
  \end{equation*}
  where $f_{j,l}^k \in A(\overline{\D})$.  (Because there are a finite
  number of functions $g_k$, the same $N$ will do every $k$.)  By the
  properties of the $A(\overline{\Omega})$-functional calculus for $T$
  we see that for $k = 1,\ldots,r$,
  \begin{equation*}
    g_k(T) = \sum_{j=1}^N f_{j,1}^k(\varphi_1(T))\cdot\,\cdots\,\cdot
    f_{j,n}^k(\varphi_n(T)).
  \end{equation*}
  Using the fact that $\overline{\D}$ is a $K'$-spectral set for
  $\varphi_k(T)$, we get
  \begin{equation*}
    \begin{split}
      \|g_k(T)\| &\leq \sum_{j=1}^N \|f_{j,1}^k(\varphi_1(T))\|\cdots
      \|f_{j,n}^k(\varphi_n(T))\| \\
      &\leq \sum_{j=1}^N (K')^n\|f_{j,1}^k\|_{A(\overline{\D})}\cdots
      \|f_{j,n}^k\|_{A(\overline{\D})}.
    \end{split}
  \end{equation*}
  This shows that for $k=1,\ldots,n$, $\|g_k(T)\| \leq C'$, with $C'$
  independent of~$T$.

  Finally, we proceed as in the proof of
  Lemma~\ref{lemma-finite-codim}.  Take $f \in (X + \mathcal{A}_\Phi)
  \otimes M_s$ and estimate
  \begin{equation*}
    \begin{split}
      \|f(T)\| &\leq \|[(G \otimes \id_s)(f)](T)\| + \sum_{k=1}^r
      \|g_k(T)\otimes [(\alpha_k \otimes \id_s)(f)]\| \\ &\leq
      C\|G\|_\cb \|f\|_{A(\overline{\Omega})\otimes M_s} +
      \sum_{k=1}^r
      C'\|\alpha_k\|_\cb\|f\|_{A(\overline{\Omega})\otimes M_s}.
    \end{split}
  \end{equation*}
  Apply Lemma~\ref{lemma-cb} to get \eqref{eq:this-lemma-*}.  The
  remaining part of the lemma now follows.
\end{proof}

We will also need a lemma that allows one to pass to the limit in a
family of inequalities of the form \eqref{eq:dm-**} depending on some
parameter $\varepsilon$.  The subspaces which play the role of $X$
will be given by the kernels of finite rank operators
$\Sigma_\varepsilon$.

\begin{lemma}
  \label{lemma-lambdas}
  Let $\{T_\varepsilon\}_{0 \leq \varepsilon \leq \varepsilon_0}
  \subset \B(H)$, with $\sigma(T_\varepsilon) \subset
  \overline{\Omega}$ for $0 \leq \eps\le\eps_0$, and
  $\{\Sigma_\varepsilon\}_{0 \leq \varepsilon < \eps_0} \subset
  \B(A(\overline{\Omega}), \C^r)$.  Assume that the maps $\varepsilon
  \mapsto T_\varepsilon$ and $\varepsilon \mapsto \Sigma_\varepsilon$
  are continuous in the norm topology.  Assume also that $\Sigma_0$ is
  surjective and that for all $s \geq 1$ and for all $\varepsilon \in
  (0,\eps_0]$,
  \begin{equation}
    \label{eq:lambdas-*}
    \|f(T_\varepsilon)\| \leq C\|f\|_{A(\overline{\Omega}) \otimes
      M_s}, \qquad
    \forall f \in (\ker \Sigma_\varepsilon \cap
    \Rat(\overline{\Omega})) \otimes M_s,
  \end{equation}
  where $C$ is a constant independent of $\varepsilon$.  Then
  \eqref{eq:lambdas-*} also holds with $\varepsilon = 0$.
\end{lemma}

\begin{proof}
  Since $\Sigma_0$ is surjective, $X = \ker \Sigma_0$ has
  codimension $r$ in $A(\overline{\Omega})$.  We apply
  Lemma~\ref{lemma-complement-polynomials} with $Y =
  \Rat(\overline{\Omega})$ to obtain functions $g_1,\ldots,g_r \in
  \Rat(\overline{\Omega})$, a subspace $Z =
  \operatorname{span}\{g_1,\ldots,g_r\}$, functionals
  $\alpha_1,\ldots,\alpha_r \in A(\overline{\Omega})^*$ and an
  operator $G$ as in the statement of that lemma.

  Consider the restrictions $\Sigma_\varepsilon | Z : Z \to \C^r$.
  The operator $\Sigma_0 | Z$ is invertible, therefore,
  $\Sigma_\varepsilon | Z$ is invertible for $\varepsilon$
  sufficiently small.  Put $P_\varepsilon =
  (\Sigma_\varepsilon|Z)^{-1}\Sigma_\varepsilon$.  Thus $P_\varepsilon
  : A(\overline{\Omega}) \to Z$ and $P_\varepsilon^2 = P_\varepsilon$.
  Indeed, $P_\varepsilon$ is the projection onto $Z$ parallel to $\ker
  \Sigma_\varepsilon$.  Define $\alpha_k^\varepsilon \in
  (A(\overline{\Omega}))^*$ by $\alpha_k^\varepsilon(f) =
  \alpha_k(P_\varepsilon f)$, and check that $G_\varepsilon(f)
  \overset{\text{def}}{=}f - \sum \alpha_k^\varepsilon(f) g_k$ is in
  $\ker \Sigma_\varepsilon$ for every $f \in A(\overline{\Omega})$.
  We compute
  \begin{equation*}
    \begin{split}
      P_\varepsilon G_\varepsilon(f) &= P_\varepsilon f - \sum_{k=1}^r
      \alpha_k^\varepsilon(f)P_\varepsilon g_k \\
      &= P_\varepsilon^2 f
      - \sum_{k=1}^r \alpha_k(P_\varepsilon f) P_\varepsilon g_k =
      P_\varepsilon G( P_\varepsilon f) = 0,
    \end{split}
  \end{equation*}
  because $P_\varepsilon f \in Z$ and $\ker G = Z$.  It follows that
  $G_\varepsilon(f)$ is in $\ker P_\varepsilon = \ker
  \Sigma_\varepsilon$.

  Since $T_\varepsilon$ depends continuously on $\varepsilon$, there
  is some constant $K$ independent of $\varepsilon$ such that
  $\|g_k(T_\varepsilon)\| \leq K$ for small $\varepsilon$ and $k =
  1,\ldots, r$.  Take $f \in \Rat(\overline{\Omega}) \otimes M_s$ and
  estimate
  \begin{equation*}
    \begin{split}
      \|f(T_\varepsilon)\| &= \Big\|[(G_\varepsilon \otimes
      \id_s)(f)](T_\varepsilon) + \sum_{k=1}^r
      g_k(T_\varepsilon)\otimes[(\alpha_k^\varepsilon \otimes \id_s)(f)]
      \Big\| \\ &\leq C\|(G_\varepsilon\otimes
      \id_s)(f)\|_{A(\overline{\Omega})\otimes M_s} + \sum_{k=1}^r
      K\|(\alpha_k^\varepsilon\otimes \id_s)(f)\|_{A(\overline{\Omega})
        \otimes M_s}.
    \end{split}
  \end{equation*}
  Since $G_\varepsilon$ and $\alpha_k^\varepsilon$ depend continuously
  on $\varepsilon$, we can let $\varepsilon \to 0$ to obtain
  \begin{equation*}
    \|f(T_0)\| \leq C\|(G \otimes
    \id_s)(f)\|_{A(\overline{\Omega})\otimes M_s} + \sum_{k=1}^r K
    \|(\alpha_k\otimes \id_s)(f)\|_{A(\overline{\Omega}) \otimes M_s}.
  \end{equation*}
  The proof concludes by noting that if $f \in \ker \Sigma_0 \otimes
  M_s$, then $(G\otimes \id_s)(f) = f$ and $(\alpha_k \otimes
  \id_s)(f) = 0$ for $k = 1,\ldots,r$.
\end{proof}

The next lemma constructs a family of admissible functions
$\Phi_\varepsilon$ which work well with the operators
$\psi_\varepsilon(T)$, where $\{\psi_\varepsilon\}$ is a shrinking for
$\Omega$.

\begin{lemma}
  \label{lemma-build-shrinking}
  Let $\Omega$ be a Jordan domain with a shrinking
  $\{\psi_\varepsilon\}_{0 \leq \varepsilon \leq \varepsilon_0}$, and
  let $\Phi : \overline{\Omega} \to \overline{\D}^n$ be admissible and
  analytic in a neighborhood of $\overline{\Omega}$.  Let $T \in
  \B(H)$ with $\sigma(T) \subset \overline{\Omega}$ and such that
  $\overline{\D}$ is a complete $K$-spectral set for $\varphi_k(T)$,
  for $k = 1,\ldots, n$.  Then there is some $0 < \delta \leq
  \varepsilon_0$ and a family of admissible functions
  $\{\Phi_\varepsilon\}_{0 \leq \varepsilon \leq \delta}$ over
  $\Omega$ with $\Phi_\varepsilon = (\varphi_1^\varepsilon, \ldots,
  \varphi_n^\varepsilon)$ and $\Phi_0 = \Phi$, such that each
  $\varphi_k^\varepsilon$ is analytic in some neighborhood $U_k$ of
  $\Omega \cup J_k$, the map $\varepsilon \mapsto
  \varphi_k^\varepsilon$ is continuous from $[0,\delta]$ to
  $C^\infty(U_k)$, and $\overline{\D}$ is a complete $K$-spectral set
  for $\varphi_k^\varepsilon(\psi_\varepsilon(T))$.
\end{lemma}

\begin{proof}
  We construct admissible functions $\Phi_\varepsilon =
  (\varphi_1^\varepsilon, \ldots, \varphi_n^\varepsilon)$ satisfying
  the statement of the lemma by choosing $\varphi_k^\varepsilon$ to
  have the form $\varphi_k^\varepsilon = \eta_k^\varepsilon \circ
  \varphi_k \circ \psi_\varepsilon^{-1}$, where $\eta_k^\varepsilon
  \in A(\overline{\D})$ and $\|\eta_k^\varepsilon\|_{A(\overline{\D})}
  \leq 1$.  Because $\varphi_k^\varepsilon(\psi_\varepsilon(T)) =
  \eta_k^\varepsilon(\varphi_k(T))$, this will guarantee that
  $\varphi_k^\varepsilon(\psi_\varepsilon(T))$ has $\overline{\D}$ as
  a complete $K$-spectral set.  The construction of
  $\eta_k^\varepsilon$ is geometric.

  First, continue analytically the arcs $J_k \subset \partial \Omega$
  to larger arcs $\widetilde{J}_k$ such that $\varphi_k$ and
  $\psi_\varepsilon$ are analytic in a neighborhood of
  $\widetilde{J}_k$ (recall that $\varphi_k$ and $\psi_\varepsilon$
  are analytic in a neighborhood of $\overline{\Omega}$).  In this
  proof, we only deal with closed arcs.  Assume that $\widetilde{J}_k$
  are small enough that each $\varphi_k|\widetilde{J}_k$ is still one
  to one.  Put $\Gamma_k^{\varepsilon} =
  \varphi_k(\psi_\varepsilon^{-1}(J_k))$ and
  $\widetilde{\Gamma}_k^{\varepsilon} =
  \varphi_k(\psi_\varepsilon^{-1}(\widetilde{J}_k))$.  Since
  $\widetilde{\Gamma}_k^0 = \varphi_k(\widetilde{J}_k)$ is an arc of
  $\T$, it follows by continuity that for small $\varepsilon$, there
  exists $\widetilde{I}_k^\varepsilon$ an arc of $\T$, and a function
  $a_k^\varepsilon : \widetilde{I}_k^\varepsilon \to \R_+$ such that
  $\widetilde{\Gamma}_k^\varepsilon = \{a_k^\varepsilon(\zeta)\zeta :
  \zeta \in \widetilde{I}_k^\varepsilon\}$.  Also, $a_k^\varepsilon
  \geq 1$ in $\widetilde{I}_k^\varepsilon$ and $a_k^0 = 1$ in
  $\widetilde{I}^0_k$.  The functions $a_k^\varepsilon$ are assumed to
  be defined for $0\le \varepsilon\le\de$.  Let $I_k^\varepsilon$ be
  the sub-arc of $\widetilde{I}_k^\varepsilon$ such that
  $\Gamma_k^\varepsilon = \{a_k^\varepsilon(\zeta)\zeta : \zeta \in
  I_k^\varepsilon\}$.

  Next find functions $b_k^\varepsilon : \T \to \R_+$, $0 \leq
  \varepsilon \leq \delta$, such that $b_k^\varepsilon \in
  C^\infty(\T)$ for each $\varepsilon$, the map $\varepsilon \mapsto
  b_k^\varepsilon$ is continuous from $[0,\delta]$ to $C^\infty(\T)$,
  $b_k^\varepsilon = a_k^\varepsilon$ in $I_k^\varepsilon$,
  $b_k^\varepsilon \geq 1$ in $\T$, and if $D_k^\varepsilon$ is the
  interior domain of the Jordan curve $\{b_k^\varepsilon(\zeta)\zeta :
  \zeta \in \T\}$, then $\varphi_k(\psi_\varepsilon^{-1}(\Omega))
  \subset \overline{D_k^\varepsilon}$.  These are first constructed in
  a local manner and then a partition of unity argument is employed.
  This construction is done as follows.

  For each $k$, define the following closed subsets of $\T \times
  [0,\delta]$:
  \begin{equation*}
    V_k = \bigcup_{0\leq\varepsilon\leq\delta}
    (I_k^\varepsilon \times \{\varepsilon\}),
    \qquad
    \widetilde{V}_k = \bigcup_{0\leq\varepsilon\leq\delta}
    (\widetilde{I}_k^\varepsilon \times \{\varepsilon\}).
  \end{equation*}
  (These are closed because $I_k^\varepsilon$ and
  $\widetilde{I}_k^\varepsilon$ depend continuously on $\varepsilon$.)
  Next, for every point $p = (\zeta,\varepsilon) \in
  \T\times[0,\delta]$ and every $k$, construct a function $c_k^p: W_p
  \to \R_+$, where $W_p$ is some neighborhood of $p$ in
  $\T\times[0,\delta]$.  If $\zeta \in I_k^\varepsilon$, choose $W_p$
  small enough so that $W_p \subset \widetilde{V}_k$ and put
  $c_k^p(\zeta',\varepsilon') = a_k^{\varepsilon'}(\zeta')$.  Note
  that if $(\zeta',\varepsilon')\in W_p$ and $r\zeta' \in \partial
  \varphi_k( \psi_{\varepsilon'}^{-1} (\Omega)) $, then $r =
  c_k^p(\zeta',\varepsilon')$.  If $\zeta \notin I_k^\varepsilon$,
  then choose $W_p$ small enough so that $W_p$ does not intersect
  $V_k$, and then choose as $c_k^p$ some $C^\infty$ function
  satisfying the property that if $(\zeta',\varepsilon') \in W_p$ and
  $r\zeta' \in \partial \varphi_k(\psi_{\varepsilon'}^{-1}(\Omega))$,
  then $r \leq c_k^p(\zeta',\varepsilon')$.  We also require $c_k^p
  \geq 1$ in all $W_p$.

  By compactness, choose a finite subfamily $\{W_{p_j}\}$ of
  $\{W_p\}$, which still covers $\T\times [0,\de]$.  Let
  $\{\tau_{p_j}\}$ be a $C^\infty$ partition of unity in $\T \times
  [0,\delta]$ subordinate to the cover $\{W_{p_j}\}$ and put
  \begin{equation*}
    b_k^\varepsilon(\zeta) = \sum_{p_j}
    \tau_{p_j}(\zeta,\varepsilon)c_k^{p_j}(\zeta,\varepsilon).
  \end{equation*}
  It is easy to see that $b_k^\varepsilon$ satisfies the required
  conditions because the functions $c_k^p$ satisfy them in a local
  manner.

  Let $D_k^\varepsilon$ be defined as above and let
  $\eta_k^\varepsilon$ be the Riemann map from $D_k^\varepsilon$ onto
  $\D$ such that $\eta_k^\varepsilon(0) = 0$ and
  $(\eta_k^\varepsilon)'(0) > 0$.  This exists since $\D \subset
  D_k^\varepsilon$.  Clearly, $\eta_k^\varepsilon \in
  A(\overline{\D})$ and $\|\eta_k^\varepsilon\|_{A(\overline{\D})}
  \leq 1$.

  We prove that $\varphi_k^\varepsilon = \eta_k^\varepsilon \circ
  \varphi_k \circ \psi_\varepsilon^{-1}$ depend continuously on
  $\varepsilon$.  Put $\beta = \max_{k,\varepsilon,\zeta}
  b_k^\varepsilon(\zeta)$, which is greater than $1$.  Let $\gamma :
  \R \to \R$ be a $C^\infty$ function such that $\gamma(r) = 0$ in a
  neighborhood of $0$, $\gamma(r) = r$ on $(\sigma,\infty)$ for some
  $\sigma\in(0,1)$ and $\gamma'(r) < \beta/(\beta-1)$ for all $r$.
  For each $\varepsilon \in [0,\delta]$, put
  \begin{equation}
    \label{eq:hk}
    h_k^\varepsilon(r\zeta) = \rho_k^\varepsilon(r)\zeta,\qquad
    \rho_k^\varepsilon(r) = r -
    \left(1-\frac{1}{b_k^\varepsilon(\zeta)}\right)\gamma(r),
    \qquad r \geq
    0,\ \zeta \in \T.
  \end{equation}
  The condition $\gamma'(r) < \beta/(\beta-1)$ implies that
  $(\rho_k^\varepsilon)' > 0$.  Thus, \eqref{eq:hk} defines maps
  $h_k^\varepsilon : \C \to \C$ which are diffeomorphisms from
  $\overline{D_k^\varepsilon}$ to $\overline{\D}$ and depend
  continuously on $\varepsilon$.  By \cite{BertrandGong}*{Corollary
    9.4}, the maps $\varepsilon \mapsto \eta_k^\varepsilon \circ
  (h_k^\varepsilon)^{-1}$ are continuous from $[0,\delta]$ to
  $C^\infty(\overline{\D})$.  Hence, the maps $\varepsilon \mapsto
  \varphi_k^\varepsilon$ are continuous from $[0,\delta]$ to
  $C^\infty(\overline{\Omega})$.

  Since by construction $|\varphi_k^\varepsilon| = 1$ in
  $\widetilde{J}_k$, the Schwartz reflection principle implies that
  each $\varphi_k^\varepsilon$ is analytic in some neighborhood $U_k$
  of $\Omega \cup J_k$ and that the map $\varepsilon \mapsto
  \varphi_k^\varepsilon$ is continuous from $[0,\delta]$ to
  $C^\infty(U_k)$.  As $\Phi_0 = \Phi$ is admissible, by continuity
  the maps $\Phi_\varepsilon$ must also be admissible for sufficiently
  small $\varepsilon$.  This finishes the proof.
\end{proof}

The following is a continuous ($\varepsilon$-dependent) version of the
right regularization for Fredholm operators of index $0$.

\begin{lemma}
  \label{lemma-regularization}
  Let $V$ be a Banach space, and $\{L_\varepsilon\}_{0\leq \varepsilon
    \leq \eps_0} \subset \B(V)$ be such that the map $\varepsilon
  \mapsto L_\varepsilon$ is continuous in the norm topology and $L_0 -
  I$ is compact.  Then there is a finite rank operator $P \in \B(V)$,
  some $0 < \delta \leq \varepsilon_0$ and operators
  $\{R_\varepsilon\}_{0 \leq \varepsilon \leq \delta},
  \{S_\varepsilon\}_{0 \leq \varepsilon \leq \delta} \subset \B(V)$
  such that the maps $\varepsilon \mapsto R_\varepsilon$ and
  $\varepsilon \mapsto S_\varepsilon$, $S_0 = I$, are continuous in
  the norm topology, and
  \begin{equation*}
    L_\varepsilon R_\varepsilon = I + P S_\varepsilon
  \end{equation*}
  holds for $0 \leq \eps \leq \delta$.
\end{lemma}

\begin{proof}
  Since $L_0 - I$ is compact, it is well know that there is a finite
  rank operator $P$ and an operator $R_0$ such that $LR_0 = I + P$.
  Let $B_\varepsilon = I + (L_\varepsilon - L_0)R_0$.  Then there is
  some $\delta > 0$ such that $B_\varepsilon$ is invertible for $0
  \leq \varepsilon \le \delta$.  We have $L_\varepsilon R_0
  B_\varepsilon^{-1} = I + PB_\varepsilon^{-1}$, so the lemma holds
  with $R_\varepsilon = R_0 B_\varepsilon^{-1}$ and $S_\varepsilon =
  B_\varepsilon^{-1}$.
\end{proof}

\begin{lemma}
  \label{lemma-0}
  Let $\Phi:\overline{\Omega}\to\overline{\D}^n$ be admissible.
  Assume that there are operators $T \in \B(H)$ and
  $C_1,\ldots,C_n\in\B(H)$ such that $\overline{\D}$ is a complete
  $K'$-spectral set for every $C_k$, $k=1,\ldots,n$.  Assume that if
  $f \in \Rat(\overline{\Omega})$ can be written as in
  \eqref{eq:C_k0}, then \eqref{eq:C_k} holds $($see the statement of
  Theorem~\ref{main2}$)$.  Then $\overline{\Omega}$ is a complete
  $K$-spectral set for $T$ for some $K$ depending on $\Omega$, $\Phi$,
  $K'$ and $S_\Lambda(T)$.  Furthermore,
  \begin{equation}
    \label{eq:this-*}
    \|f(T)\| \leq C \|f\|_{A(\overline{\Omega})\otimes M_s},
    \qquad \forall f \in (\mathcal{A}_\Phi \cap
    \Rat(\overline{\Omega})) \otimes M_s,
  \end{equation}
  where $C$ is a constant depending only on $\Omega$, $\Phi$ and $K'$,
  and not on $T$.
\end{lemma}

The main point of \eqref{eq:this-*} is that, under the hypotheses of
this Lemma, $\mathcal{A}_\Phi$ is a closed subspace of finite
codimension in $A(\overline{\Omega})$.  Thus, \eqref{eq:this-*} shows
that, in a space of finite codimension, the inequality $\|f(T)\|\leq
C\|f\|$ holds with a constant independent of $T$.

\begin{proof}[Proof of Lemma~\ref{lemma-0}]
  Use Theorem~\ref{article1-thm1} to obtain operators $F_k$ as in the
  statement of the theorem.  Denote by $L \in
  \B(A(\overline{\Omega}))$ the operator defined by $L(f) = \sum
  F_k(f)\circ \varphi_k$.  Since $I-L$ is compact, there exist an
  operator $R$ and a finite rank operator $P$ such that $LR = I + P$.
  The space $X = \ker P$ has finite codimension in
  $A(\overline{\Omega})$ and does not depend on $T$.  We will now
  check that \eqref{eq:dm-**} holds for some constant $C$ independent
  of $T$.

  Take $f \in (X \cap \Rat(\overline{\Omega})) \otimes M_s$ and put $g
  = (R\otimes \id_s)f$.  Then $(L \otimes \id_s)g = f$, and so
  by~\eqref{eq:C_k},
  \begin{equation*}
  f(T) = \sum_{k=1}^n [(F_k\otimes \id_s)(g)](C_k).
  \end{equation*}
  Since $\overline{\D}$ is complete $K'$-spectral for $C_k$,
  \begin{equation*}
  \begin{split}
    \|f(T)\| &\leq \sum_{k=1}^n K'\|F_k\|_\cb
    \|g\|_{A(\overline{\Omega})\otimes M_s} \leq \sum_{k=1}^n
    K'\|F_k\|_\cb \|R\|_\cb \|f\|_{A(\overline{\Omega})\otimes M_s}
    \\
    &= \sum_{k=1}^n K'\|F_k\|\cdot \|R\|\cdot
    \|f\|_{A(\overline{\Omega})\otimes M_s},
  \end{split}
  \end{equation*}
  where the last equality uses Lemma~\ref{lemma-cb}.  Thus
  \eqref{eq:dm-**} holds with $C = \sum K'\|F_k\|\cdot\|R\| < \infty$.
  Apply Lemma~\ref{lemma-finite-codim-Aphi} to get \eqref{eq:this-*}.
  The remaining part of the lemma follows from
  Lemma~\ref{lemma-finite-codim}.
\end{proof}

\section{Proofs of Theorems 3 and 4}
\label{proofs-main}

We first give the proof of Theorem~\ref{main2}, as it is simpler than
that of Theorem~\ref{main} and both proofs follow the same general
idea.

\begin{proof}[Proof of Theorem~\ref{main2}]
  The first part of Theorem~\ref{main2} is already contained in
  Lemma~\ref{lemma-0}.  For the case when $\Phi$ is injective and
  $\Phi'$ does not vanish, use Theorem~\ref{article1-thm2} (see
  Section~\ref{sec:article1}).  Then \eqref{eq:this-*} implies that
  $\overline{\Omega}$ is a complete $K$-spectral set for $T$, with $K$
  independent of~$T$.
\end{proof}

To prove Theorem~\ref{main}, in the case when $\sigma(T) \subset
\Omega$, one can argue as in the proof of Theorem~\ref{main2}, putting
$C_k = \varphi_k(T)$ and using the Cauchy-Riesz functional calculus
for $T$ to get \eqref{eq:C_k}.  However, such a direct proof will not
work in the general case.  The idea then is to apply a shrinking
$\{\psi_\varepsilon\}$ for $\Omega$ to obtain operators $T_\varepsilon
= \psi_\varepsilon(T)$ which have $\sigma(T_\varepsilon) \subset
\Omega$, so that the above argument is again valid.  The difficulties
reside in constructing admissible functions $\Phi_\varepsilon =
(\varphi_1^\varepsilon,\ldots,\varphi_n^\varepsilon)$ adapted to
$T_\varepsilon$, in the sense that each
$\varphi_k^\varepsilon(T_\varepsilon)$ has $\overline{\D}$ as a
complete $K'$-spectral set, as well as in passing to the limit as
$\varepsilon$ tends to $0$.

\begin{proof}[Proof of Theorem~\ref{main}]
  Let $T \in \B(H)$ with $\sigma(T) \subset \overline{\Omega}$ and
  such that for $k = 1,\ldots,n$, $\overline{\D}$ is a complete
  $K'$-spectral set for $\varphi_k(T)$.  We must prove that
  $\overline{\Omega}$ is a complete $K$-spectral set for $T$ with $K$
  depending on $\overline{\Omega}$, $\Phi$, $K'$ and $S_\Lambda(T)$.

  By Lemma~\ref{lemma-shrinking}, there is a shrinking
  $\{\psi_\varepsilon\}$ for $\Omega$.  Apply
  Lemma~\ref{lemma-build-shrinking} to obtain a collection of
  admissible functions $\{\Phi_\varepsilon\}_{0 \leq \varepsilon \leq
    \varepsilon_0}$ such that $\overline{\D}$ is a complete
  $K'$-spectral set for $\varphi_k^\varepsilon(\psi_\varepsilon(T))$,
  and such that the maps $\varepsilon \mapsto \varphi_k^\varepsilon$
  are continuous from $[0,\varepsilon_0]$ to $C^\infty(U_k)$, where
  $U_k$ is a neighborhood of $J_k$.  Then use
  Lemma~\ref{article1-lemma} to get operators $L_\varepsilon$, where
  \begin{equation*}
    L_\varepsilon(f) = \sum F_k^\varepsilon(f)\circ
    \varphi_k^\varepsilon.
  \end{equation*}
  Since $L_0 - I$ is compact, Lemma~\ref{lemma-regularization} (with
  $V = A(\overline{\Omega})$) yields operators
  $P,R_\varepsilon,S_\varepsilon:A(\overline{\Omega})\to
  A(\overline{\Omega})$, where $\varepsilon\in [0,\delta]$, with the
  properties stated in the lemma.

  Next we wish to apply Lemma~\ref{lemma-lambdas}.  To this end, fix
  $Q : \operatorname{ran} P \to \C^r$ an isomorphism, where $r$ is the
  rank of $P$, put $\Sigma_\varepsilon = QPS_\varepsilon$, so that
  $\Sigma_\varepsilon : A(\overline{\Omega}) \to \C^r$,
  $\Sigma_\varepsilon$ depends continuously on $\varepsilon$ in the
  norm topology and $\Sigma_0 = QPS_0 = QP$ is surjective, and set
  $T_\varepsilon = \psi_\varepsilon(T)$.  Note that $T_\varepsilon$
  depends continuously on $\varepsilon$ in the norm topology because
  $\psi_\varepsilon$ depends continuously on $\varepsilon$ in the
  topology of uniform convergence on compact subsets of $U$, where $U$
  is some open neighborhood of $\sigma(T)$.

  It is necessary to check that \eqref{eq:lambdas-*} holds.  For this,
  take $f \in (\ker \Sigma_\varepsilon \cap \Rat(\overline{\Omega}))
  \otimes M_s$, put $g = (R_\varepsilon \otimes \id_s)f$ and note that
  $f = (L_\varepsilon \otimes \id_s)g$.  Since $\sigma(T_\varepsilon)
  \subset \Omega$, an application of the Cauchy-Riesz functional
  calculus gives
  \begin{equation*}
  f(T_\varepsilon) = \sum_{k=1}^n [(F_k^\varepsilon\otimes
  \id_s)(g)](\varphi_k^\varepsilon(T_\varepsilon)).
  \end{equation*}
  Therefore, by Lemma~\ref{lemma-cb}, and since $\|F_k^\varepsilon\|
  \leq C$ (coming from Lemma~\ref{article1-lemma}),
  \begin{equation*}
    \|f(T_\varepsilon)\| \leq \sum_{k=1}^n
    K'\|F_k^\varepsilon\|_\cb\|g\|_{A(\overline{\Omega}) \otimes M_s}
    \leq \sum_{k=1}^n K'C\|R_\varepsilon\|
    \|f\|_{A(\overline{\Omega})\otimes M_s}.
  \end{equation*}
  Since $R_\varepsilon$ depends continuously on $\varepsilon$,
  \eqref{eq:lambdas-*} holds, as desired.

  Apply Lemma~\ref{lemma-lambdas} to obtain for all $s\geq 1$,
  \begin{equation*}
    \|f(T)\| \leq C'\|f\|_{A(\overline{\Omega})\otimes M_s}, \qquad
    \forall f\in (\ker \Sigma_0 \cap \Rat(\overline{\Omega}))\otimes
    M_s.
  \end{equation*}
  By Lemma~\ref{lemma-finite-codim}, this yields that
  $\overline{\Omega}$ is a complete $K$-spectral set for $T$, with $K$
  depending on $\Omega$, $\Phi$ and $S_\Lambda(T)$.  Therefore, $\Phi$
  is a quasi-uniform strong test collection.

  In the case that $\Phi$ is injective and $\Phi'$ does not vanish on
  $\Omega$, Theorem~\ref{article1-thm2} and
  Lemma~\ref{lemma-finite-codim-Aphi} together imply that
  $\overline{\Omega}$ is a complete $K$-spectral set for $T$, with $K$
  independent of $T$.
\end{proof}

\section{Weakly admissible functions}
\label{weakly-admissible}

In this section we will expand the class of functions under
consideration to a wider class that we call weakly admissible
functions.  The main goal of this class is to replace condition (f) in
the definition of an admissible function (see
Section~\ref{subsec:admiss-funct-famil}) by a weaker separation
condition.  In particular, a collection of functions which includes
inner functions (i.e., functions with modulus 1 in all
$\partial\Omega$) may be weakly admissible, though not admissible,
except in trivial cases.

Let $\zeta \in \partial\Omega$.  A \emph{right neighborhood} of
$\zeta$ in $\partial\Omega$ is understood to be the image
$\gamma([0,\varepsilon))$, where the function $\gamma : [0,\epsilon)
\to \partial \Omega$ is continuous and injective, $\gamma(0) = \zeta$,
and as $t$ increases $\gamma(t)$ follows the positive orientation of
$\partial\Omega$.  Define the \emph{left neighborhoods} of $\zeta$ in
a similar manner.

If $\Psi \subset A(\overline{\Omega})$ is a collection of functions
taking $\Omega$ into $\D$ and $\zeta \in \partial\Omega$, set
\begin{equation*}
  \Psi_\zeta^+ = \{ \psi \in \Psi : |\psi| = 1 \text{ in some right
    neighborhood of } \zeta\},
\end{equation*}
and
\begin{equation*}
  \Psi_\zeta^- = \{ \psi \in \Psi : |\psi| = 1 \text{ in some left
    neighborhood of } \zeta\}.
\end{equation*}

\begin{definition}
  Let $\Omega$ be a domain whose boundary is a disjoint finite union
  of piecewise analytic Jordan curves such that the interior angles of
  the corners of $\partial\Omega$ are in $(0,\pi]$.  Then $\Psi =
  (\psi_1,\ldots,\psi_n) : \overline{\Omega} \to \overline{\D}^n$,
  $\psi_k \in A(\overline{\Omega})$ for $k = 1,\ldots,n$, is
  \emph{weakly admissible} if for $\Gamma_k = \{\zeta \in \partial
  \Omega: |\psi_k(\zeta)| = 1\}$ in place of $J_k$ and a constant
  $\alpha$, $0 < \alpha \leq 1$, it is the case that conditions
  (a)--(e) for an admissible function hold, and additionally:
  \begin{enumerate}[(a)]
  \item[(f$'$)] \quad $\forall \zeta \in \partial\Omega,\ \forall z
    \in \partial \Omega, z\neq\zeta, \ \exists\psi \in \Psi_\zeta^+ :
    \psi(\zeta)\neq\psi(z).$
  \item[(g$'$)] \quad $\forall \zeta \in \partial\Omega,\ \forall z
    \in \partial \Omega, z\neq\zeta, \ \exists\psi \in \Psi_\zeta^- :
    \psi(\zeta)\neq\psi(z).$
  \end{enumerate}
\end{definition}

In fact, it is easy to see that conditions (a) and (b) follow formally
from conditions (c)--(e), (f$'$) and~(g$'$).

\begin{lemma}
  \label{passing}
  Let $\Psi : \overline{\Omega} \to \overline{\D}^n$ be a weakly
  admissible function.  Then there is an admissible function $\Phi :
  \overline{\Omega} \to \overline{\D}^m$, $\Phi =
  (\varphi_1,\ldots,\varphi_m)$, such that its components $\varphi_k$
  are of the form
  \begin{equation*}
    \varphi_k =
    (h_{1,k}\circ\psi_1)\cdot\,\cdots\,\cdot(h_{n,k}\circ\psi_n),
  \end{equation*}
  where $h_{j,k} : \overline{\D} \to \overline{\D}$ and $h_{j,k} \in
  A(\overline{\D})$.
\end{lemma}

\begin{proof}
  First, fix some $\zeta \in \partial\Omega$.  For each $\psi \in
  \Psi_\zeta^+$, put $P_\psi = \psi^{-1}(\{\psi(\zeta)\})$, which is a
  finite set of points of $\partial \Omega$.  By condition (f$'$) in
  the definition of a weakly admissible function, $\bigcap_{\psi\in
    \Psi_\zeta^+} P_\psi = \{\zeta\}$.  Let $J_\zeta^+$ be the closure
  of a sufficiently small right neighborhood of $\zeta$.  For $\psi
  \in \Psi_\zeta^+$, put $Q_\psi = \psi^{-1}(\psi(J_\zeta^+))$.  If
  $J_\zeta^+$ is small enough, then each set $Q_\psi$ is a union of
  disjoint right neighborhoods of each of the points in $P_\psi$.
  Since $\bigcap_{\psi\in \Psi_\zeta^+} P_\psi = \{\zeta\}$, it can
  then be assumed that $\bigcap_{\psi\in \Psi_\zeta^+} Q_\psi =
  J_\zeta^+$.

  Next, for each $\psi \in \Psi_\zeta^+$, construct a function
  $h_\psi^+ \in A(\overline{\D})$ such that the function
  \begin{equation*}
  \psi_\zeta^+ = \prod_{\psi \in \Psi_\zeta^+} h_\psi^+ \circ \psi
  \end{equation*}
  associated with $J_\zeta^+$ satisfies $|\psi_\zeta^+| = 1$ in
  $J_\zeta ^+$ and $|\psi_\zeta^+|< 1$ in $\partial\Omega\setminus
  J_\psi^+$.  This is done as follows.

  Take $\psi \in \Psi_\zeta^+$.  Choose
  a function
  $h_\psi^+$
  satisfying the following conditions:
  \begin{itemize}
  \item $|h_\psi^+| = 1$ in $\psi(J_\zeta^+)$ and $|h_\psi^+| < 1$ in
    $\partial\Omega\setminus \psi(J_\zeta^+)$;
  \item $h_\psi^+$ maps $\psi(J_\zeta^+)$ bijectively onto a small arc
    of $\T$;
  \item $h_\psi^+$ is analytic on some open set $U \supset \D$ such
    that the interior of $\psi(J_\zeta^+)$ relative to $\T$ is
    contained in $U$, and $(h_\psi^+)'$ is H\"older $\alpha$ in $U$;
  \item $|(h_\psi^+)'| \geq C > 0$ in $\psi(J_\zeta^+)$;
  \item If $\zeta$ is an endpoint of the set $\{w\in\partial\Omega:
    |\psi(w)|=1\}$ and $S(\zeta)$ is the sector that appears on
    condition (d) in the definition of an admissible function (for
    $\varphi_k = \psi$), then $\psi(S_k(\zeta))\subset U$.
  \end{itemize}

  Then $|h_\psi^+ \circ \psi| = 1$ in $Q_\psi$, and $|h_\psi^+ \circ
  \psi| < 1$ in $\partial\Omega\setminus Q_\psi$.  Since
  $|\psi_\zeta^+(z)| = 1$ only when $|h_\psi^+(\psi(z))|=1$ for every
  $\psi\in\Psi_\zeta^+$ (that is, when $z \in \bigcap_{\psi\in
    \Psi_\zeta^+} Q_\psi = J_\zeta^+$), we get that $|\psi_\zeta^+| =
  1$ in $J_\zeta^+$ and $|\psi_\zeta^+| < 1$ in $\partial\Omega
  \setminus J_\zeta^+$.  Also, since $h_\psi^+(\psi(J_\zeta^+))$ is a
  small arc of $\T$, it follows that $\psi_\zeta^+$ maps $J_\zeta^+$
  bijectively onto some arc of $\T$.

  Similarly, construct an arc $J_\zeta^-$ which is the closure of a
  small left neighborhood of $\zeta$, and a corresponding function
  $\psi_\zeta^-$.  By compactness, we can choose a finite set of
  points $\zeta_1,\ldots,\zeta_r$ such that $J_{\zeta_k}^- \cup
  J_{\zeta_k}^+$, $k=1,\ldots,r$, cover all $\partial \Omega$.  Rename
  the functions $\psi_{\zeta_1}^-, \psi_{\zeta_1}^+, \ldots,
  \psi_{\zeta_r}^-, \psi_{\zeta_r}^+$ as $\varphi_1,\ldots,\varphi_m$
  and the corresponding arcs $J_{\zeta_1}^-,J_{\zeta_1}^+, \ldots,
  J_{\zeta_r}^-, J_{\zeta_r}^+$ as $J_1,\ldots,J_m$.  Functions
  $\varphi_1,\ldots,\varphi_m$ now satisfy condition (f) in the
  definition of an admissible family, because if, for instance, $J_k=
  J_\zeta^+$, then $\varphi_k = \psi_\zeta^+$ sends $J_\zeta^+$
  bijectively onto an arc of $\T$ and $|\varphi_k| < 1$ on
  $\partial\Omega\setminus J_\zeta^+$.

  The functions $\varphi_k$ satisfy conditions (c)--(e) because the
  functions $\psi_k$ satisfy these conditions, and the functions
  $h_\psi^+, h_\psi^-$ satisfy similar regularity conditions which
  have been given above.  It follows that $\Phi =
  (\varphi_1,\ldots,\varphi_m)$ is admissible.
\end{proof}

\begin{theorem}
  \label{weak-main}
  Let $\Psi: \overline{\Omega} \to \overline{\D}^n$ be a weakly
  admissible function.  Then $\Psi$ is a quasi-uniform test collection
  over $\Omega$.  Moreover, if $\Psi_0 \subset A(\overline{\Omega})$
  is any collection of functions taking $\Omega$ into $\D$ with the
  property that $\Psi \subset \Psi_0$, $\Psi_0 :
  \overline{\Omega}\to\overline{\D}^m$ is injective and $\Psi_0'$ does
  not vanish on $\Omega$, then $\Psi_0$ is a uniform test collection
  over $\Omega$.
\end{theorem}

\begin{proof}
  Suppose that $T \in \B(H)$, $\sigma(T) \subset \Omega$ and
  $\psi_k(T)$ are contractions for $k=1,\ldots,n$.  Let $\Phi$ be the
  admissible function obtained from $\Psi$ using Lemma~\ref{passing}.
  Put $C_k = \varphi_k(T)$, $k = 1,\ldots,m$.  Check that $C_k$ is a
  contraction for all $k$.

  Since $\varphi_k = (h_{1,k} \circ\psi_1) \cdot \ldots \cdot (h_{n,k}
  \circ \psi_n)$,
  \begin{equation*}
    C_k = \varphi_k(T) = h_{1,k}(\psi_1(T))\cdot\,\cdots\,\cdot
    h_{n,k}(\psi_n(T)).
  \end{equation*}
  Each $\psi_j(T)$ is a contraction, and
  $\|h_{j,k}\|_{A(\overline{\D})} \leq 1$, so $h_{j,k}(\varphi_j(T))$
  is also a contraction for all $j$ and $k$.  It follows that, being a
  product of contractions, $\psi_k(T)$ is a contraction.

  Because $\sigma(T)\subset \Omega$, the hypotheses of
  Theorem~\ref{main2} are satisfied by the Cauchy-Riesz functional
  calculus for $T$.  Therefore, Lemma~\ref{lemma-0} applies, and so
  $\overline{\Omega}$ is a complete $K$-spectral set for $T$ with $K =
  K(\Omega, \Psi, S_\Lambda(T))$, for an arbitrary pole set $\Lambda$
  for $\Omega$.  In other words, $\Psi$ is a quasi-uniform test
  collection over~$\Omega$.

  Now assume that $\Psi_0$ is as in the statement of the theorem.
  Form an admissible function $\Phi_0$ from $\Phi$ by adding to $\Phi$
  all the functions in $\Psi_0\setminus\Phi$. The arcs $J_k$
  corresponding to functions in $\Psi_0\setminus\Phi$ are defined to
  be equal to the empty set.  Then $\Phi_0$ is injective and $\Phi_0'$
  does not vanish, because $\Psi_0$ already had these properties.
  Therefore, $\Phi_0$ is a uniform strong test collection over
  $\Omega$ by Corollary~\ref{main-corollary}.  If $\psi(T)$ is a
  contraction for every $\psi \in \Psi_0$, then $\varphi(T)$ is also a
  contraction for every $\varphi\in\Phi_0$.  Hence, $\Psi_0$ is a
  uniform test collection over $\Omega$.
\end{proof}

Unfortunately, the methods of the above proof cannot be used to show
that $\Psi$ is a strong test collection over $\Omega$, and we do not
know whether the hypotheses imply it.  If the operators $\psi_k(T)$,
$k = 1,\ldots,n$, are contractions, then it follows that
$\varphi_k(T)$, $k=1,\ldots,m$, is a product of contractions and
therefore a contraction.  However, if $\psi_k(T)$, $k=1,\ldots,n$,
just have $\overline{\D}$ as a complete $K$-spectral set for some $K$,
then we only get that $\varphi_k(T)$ is a product of operators which
have $\overline{\D}$ as a complete $K$-spectral set.  In general, an
operator which is the product of two commuting operators both similar
to contractions need not itself be similar to a contraction;
see~\cite{Pisier}.  Therefore, one cannot prove by this method that
$\varphi_k(T)$ has $\overline{\D}$ as a complete $K'$-spectral set for
some $K'$.

\goodbreak

\begin{corollary}
  \label{inner}
  Let $\Omega$ be a finitely connected domain with analytic boundary
  and let $\psi_1,\ldots,\psi_n : \overline{\Omega} \to \overline{\D}$
  be inner $($i.e., $|\psi_j| = 1$ in $\partial\Omega$ for
  $j=1,\ldots,n)$.  Assume that the restriction of the map
  $\Psi=(\psi_1,\ldots,\psi_n) : \overline{\Omega} \to
  \overline{\D}^n$ to $\partial\Omega$ is injective.  Then $\Psi$ is a
  quasi-uniform test-collection over $\Omega$.  If, moreover, $\Psi$
  is injective in $\overline{\Omega}$ and $\Psi'$ does not vanish in
  $\Omega$, then $\Psi$ is a uniform test collection over $\Omega$.
\end{corollary}

\begin{proof}
  Since $\psi_1,\ldots,\psi_n$ are inner, $\Psi_\zeta^- = \Psi_\zeta^+
  = \Psi$ for all $\zeta\in\partial\Omega$.  Therefore, the conditions
  (f$'$) and (g$'$) in the definition of a weakly admissible are
  equivalent to the condition that $\Psi|\partial\Omega$ is injective.
  Since $\psi_1,\ldots,\psi_n$ are inner and $\partial\Omega$ is
  analytic, $\psi_1,\ldots,\psi_n$ can be extended analytically across
  $\partial\Omega$.  Hence, $\Psi$ is a weakly admissible function.
  To finish the proof, apply Theorem~\ref{weak-main} with $\Psi_0 =
  \Psi$.
\end{proof}

On a general finitely connected domain $\Omega$ with analytic
boundary, one can always choose three inner functions
$\psi_1,\psi_2,\psi_3$ such that the map
$\Psi=(\psi_1,\psi_2,\psi_3):\overline{\Omega}\to\overline{\D}^3$ is
injective and $\Psi'$ does not vanish in $\Omega$.  Hence, such $\Psi$
is a uniform test collection according to Corollary~\ref{inner}.  See
\cite{Stout}*{Theorem IV.1} and \cite{Fedorov}*{\S{}3} for two
different proofs of the existence of such a $\Psi$.  It is also known
that when $\Omega$ is doubly connected then the same can be done using
only two inner functions $\psi_1,\psi_2$.  However, for a domain
$\Omega$ of connectivity greater or equal than 3, a pair of inner
functions $\psi_1,\psi_2$ will never be enough under the constraint
that $\Psi$ is injective (see \cites{Rudin,Fedorov}).

\end{document}